\documentclass[12pt]{amsart}

\usepackage{amsthm}

\usepackage{epsf}

\usepackage{amssymb}

\usepackage{latexsym}

\usepackage{amsmath}

\newtheorem{theorem}{Theorem} 

\newtheorem{definition}[theorem]{Definition}

\newtheorem{lemma}[theorem]{Lemma}

\newtheorem{proposition}[theorem]{Proposition}

\newtheorem{remark}[theorem]{Remark}

\DeclareMathOperator{\Ker}{Ker}

\newcommand{\wt}{\mathop {\rm wt}}

\def\R{{\mathbb R}}

\def\C{{\mathbb C}}

\def\Z{{\mathbb Z}}

\def\P{{\mathbb P}}

\def\semidirect{\rtimes}

\def \ta{\tau}

\def \ta1{\tau_1}

\def \dl{\delta}

\def \g{\gamma}

\def \G{\Gamma}

\def \wt{\widetilde}

\def \E{{\mathcal E}}

\def \O{{\mathcal O}}

\newcommand\Gal[1]{{{#1}_{\operatorname{Gal}}}}

\def \prodl{\prod\limits}

\newcommand\set[1]{{\{{#1}\}}}

\def\st{{such that }}

\newcommand\begintable[1][] {{}}

\long\def\forget#1\forgotten{}

\newif\ifXY 

\XYtrue     

\def\Z{\mathbb{Z}}

\begin{document}

\renewcommand{\subjclassname}{%

      \textup{2000} Mathematics Subject Classification}

\date{\today}

\title[Fund. Groups of Galois Covers of Small Degree]{Classification of Fundamental Groups of Galois Covers\\
 of Surfaces of Small Degree Degenerating\\ to Nice Plane Arrangements}

\author[Amram, Lehman, Shwartz, Teicher]{Meirav Amram$^{1}$, Rebecca Lehman, Robert Shwartz, Mina Teicher}

\stepcounter{footnote}

\footnotetext{Partially supported by the Emmy Noether Research Institute for Mathematics (center of the Minerva
Foundation of Germany), the Excellency Center "Group Theoretic Methods in the Study of Algebraic Varieties" of
the Israel Science Foundation, and EAGER (EU network, HPRN-CT-2009-00099).}

\address{Meirav Amram, Rebecca Lehman, Robert Shwartz, Mina Teicher, Department of Mathematics,
Bar-Ilan University, Ramat-Gan 52900, Israel}

\email{meirav/shwartr1/teicher@macs.biu.ac.il, rclehman@alum.mit.edu}

\begin{abstract}
Let $X$ be a surface of degree $n$, projected onto $\mathbb{CP}^2$. The surface has a natural Galois cover with
Galois group $S_n.$  It is possible to determine the fundamental group of a Galois cover from that of the
complement of the branch curve of $X.$ In this paper we survey the fundamental groups of Galois covers of all
surfaces of small degree $n \leq 4$, that degenerate to a nice plane arrangement, namely a union of $n$ planes
such that no three planes meet in a line.  We include the already classical examples of the quadric, the
Hirzebruch and the Veronese surfaces and the degree $4$ embedding of $\mathbb{CP}^1 \times \mathbb{CP}^1,$ and
also add new computations for the remaining cases: the cubic embedding of the Hirzebruch surface $F_1$, the
Cayley cubic (or a smooth surface in the same family), for a quartic surface that degenerates to the union of a
triple point and a plane not through the triple point, and for a quartic $4$-point. In an appendix, we also
include the degree $8$ surface $\mathbb{CP}^1\times \mathbb{CP}^1$ embedded by the $(2,2)$ embedding, and the
degree $2n$ surface embedded by the $(1,n)$ embedding, in order to complete the classification of all embeddings
of $\mathbb{CP}^1 \times \mathbb{CP}^1,$ which was begun in \cite{15}.
\end{abstract}


\maketitle

\section{Introduction}\label{outline}

\subsection{Background}

Algebraic surfaces are classified by discrete and continuous invariants. Fixing the discrete invariants gives a
family of algebraic surfaces parameterized by algebraic variety called the moduli space. All surfaces in the
same moduli space have the same homotopy type and therefore the same fundamental group. So, fundamental groups
are discrete invariants of the surfaces and form a central tool in their classification. We have no algorithm to
compute the fundamental group of a given projective algebraic surface $X,$ but we can cover $X$ by a surface
$X_{Gal}$ with a computable fundamental group.  The surface $X_{Gal}$ is the Galois cover of $X,$ and its
fundamental group is a crucial invariant of the surface $X.$  It is connected with the fundamental group of the
complement of the branch curve of a generic projection of the surface $X.$

One common solution to this problem is to degenerate $X$ to a union of planes, whose branch curve is a line
arrangement in the plane.  The fundamental groups of line arrangements are easily computable and can be
exploited to obtain information on the original branch curve and on the original surface. The idea of using
degenerations for these purposes appears in \cite{3}, \cite{10}, \cite{12}, \cite{15} and \cite{18}.
Degenerations of $K3$ surfaces were constructed in \cite{10}, \cite{11} and \cite{12}, and are called {\em
pillow degenerations}.

In recent years, braid groups have been becoming more and more popular in different branches of mathematics, such as
the theory of Jones polynomials (which is a current interest of theoretical physicists).  Former studies show
that some nontrivial properties of braid groups are intimately connected with the existence of nontrivial
geometrical objects of complex geometry and algebraic surfaces. Thus, our study of algebraic surfaces and their
branch curves gives new insight in the field of braid groups and vice versa.

The applications of the braid group technique to the study of algebraic curves in general and to the topology of complements of
curves in particular, started with  Artin in \cite{1} and \cite{2}; see also Chisini \cite{9} and Van Kampen
\cite{vk}. The braid monodromy technique is due to ideas of Chisini, Enriques, Zariski and van Kampen in the
30's (see \cite{13}, \cite{vk}, \cite{Z1}, \cite{zar}, \cite{Z3}).  It was revived by Moishezon in the late
1970's,and was first presented in a complete form by Moishezon in \cite{Mo} and \cite{Mo0}, and
Moishezon-Teicher in \cite{13}, \cite{16} and \cite{17}. Many examples of computations of braid monodromy have
been executed, see for example in \cite{3}, \cite{5}, \cite{17}, \cite{MoTe10}. Also, in \cite{dennis} one can
find a description of computations of braid monodromy and the fundamental group $\pi_1(\C\P^2-S)$ of the
complement of the branch curve of a surface.

In \cite{dennis}, the authors also considered the branch curve of a projection of a non-prime $K3$ surface. But
they considered quotients of $\pi_1(\C\P^2-S),$ where $S$ denotes the branch curve, by a subgroup of commutators
(commutators of geometric generators which are mapped to disjoint transpositions by the geometric monodromy
representation, see Definition 2.2 in that paper). Their motivation came from the theory of symplectic manifolds
and families of projections where the branch curves acquire and may lose pairs of transverse double points with
opposite orientations; creating a pair of double points adds a commutation relation. The quotient there is the
largest possible quotient of $\pi_1(\C\P^2-S)$.

Several other small cases have been considered individually: in \cite{Cayley}, the authors compute the
fundamental group of the complement of the branch curve of the Cayley cubic surface, and in \cite{Ogata}, they
compute it for the $(1,1)$ embedding of the Hirzebruch surface $F_2.$  (\cite{Robb} also computes the
fundamental group of the Galois cover of complete intersections.)   In \cite{fr1} and \cite{fr2}, the authors
compute the fundamental group of the complement of the branch curve for the second Hirzebruch surface and for
the product of a projective line and a torus, respectively.

Fundamental groups of Galois covers were first considered in \cite{3} \cite{pil}, \cite{8} and \cite{TxT}.  In
\cite{pil}, we encounter new types of singularities, namely $3$-points, local intersection points of three
distinct lines, which had not been handled before, and whose analysis is necessary for the precise computations
of the braid monodromy. Moreover, since the monodromies related to $6$-points (first discussed in \cite{17} and
\cite{19}) are quite hard to follow, \cite{pil} presents them in a precise way algebraically, accompanied by
figures illustrating the computations. \cite{8} deals with the Hirzebruch surface $F_1,$ \cite{MRT} and
\cite{FRT} deal with covers of Hirzebruch surfaces, while \cite{TxT} deals with the product of two tori, whose
branch curve was first investigated in \cite{5}. Robb \cite{Robb} computed the fundamental groups for complete
intersection surfaces. The $4$-point appears first in a local computation in \cite{17}; we show it here in the
context of an explicit algebraic surface.

The goals of this paper are twofold: first, to survey the fundamental groups of the complements of the branch
curves and of the Galois covers of all degenerations of surfaces of low degrees (less than 5), and second, to
complete the classification of Galois covers of $\C\P^1 \times \C\P^1,$ which was begun in \cite{Mo}, \cite{Mo0}
and \cite{MoTe6}.

\subsection{Method}
In order to compute the possible fundamental group of the Galois cover of a surface of small degree with at
worst isolated singularities and with a nice degeneration, we degenerate the surface to a union of planes, and
then apply the regeneration process. The algebra is independent of the specific choice of surface, as it is
determined entirely combinatorially. Therefore, to construct the set of all possible such fundamental groups of
Galois covers of surfaces of degrees $2,$ $3$ and $4,$ we construct combinatorial representations of each of the
possible admissible arrangements of two, three or four planes, i.e., those in which no two planes meet in a
line, and we verify that each of them in fact corresponds to a degeneration of some algebraic surface.  Note that
there are many possible plane arrangements corresponding to each diagram; however, the fundamental groups of the
complement of the branch curve and of the Galois cover are determined combinatorially, and are thus independent of the specific arrangement.  We then
calculate the braid monodromy factorization and the fundamental groups using the method of Moishezon-Teicher. We
present the computations in some detail for the convenience of the reader, and so as to make all the examples
completely explicit.  The computations proceed in three steps.

First, we compute the braid monodromy factorization of the branch curve, as defined in \cite[Prop. VI.2.1]{16}.

Consider the following setting (Figure \ref{setup}). Let $S$ be an algebraic curve in $\C^2$, with $p = \deg(S)$.
Let $\pi: \C^2 \rightarrow \C$ be a generic projection onto the first coordinate. Define the fiber $K(x) = \{y
\mid (x,y) \in S\}$ in $S$ over a fixed point $x$, projected to the $y$-axis. Define $N = \{x \mid \# K(x) < p
\}$ and $M' = \{ s \in S \mid \pi_{\mid s} \mbox{ is not \'{e}tale at } s \}$; note that $\pi (M') = N$. Let
$\set{A_j}^q_{j=1}$ be the set of points of $M'$ and $N = \set{x_j}^q_{j=1}$ their projections onto the $x$-axis.
Recall that $\pi$ is generic, so we assume that $\# (\pi^{-1}(x) \cap M') =1$ for every $x \in N$. Let $E$
(resp. $D$) be a closed disk on the $x$-axis (resp. the $y$-axis), \st $M' \subset E \times D$ and $N \subset
\mbox{Int}(E)$. We choose $u \in \partial E$ a real point far enough from the set $N$, so that $x << u$ for
every $x \in N$. Define $\C_u = \pi^{-1}(u)$ and number the points of $K=\C_u\cap S$ as $\{1 , \dots , p\}$.

\begin{figure}[h]
\begin{minipage}{\textwidth}
\begin{center}
\epsfbox{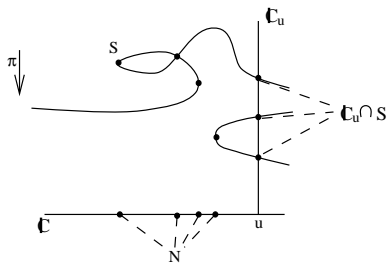}
\end{center}
\end{minipage}
\caption{General setting}\label{setup}
\end{figure}

We now construct a g-base for the fundamental group $\pi_1(E - N, u)$. Take a set of paths $\set{\g_j}^q_{j=1}$
which connect $u$ with the points $\set{x_j}^q_{j=1}$ of $N$. Now encircle each $x_j$ with a small oriented
counterclockwise circle $c_j$. Denote the path segment from $u$ to the boundary of this circle by $\g'_j$. We
define an element (a loop) in the g-base as $\delta_j = {\g'}_j c_j {\g'}^{-1}_j$. Let $B_p[D,K]$ be the braid
group, and let $H_1, \dots , H_{p-1}$ be its frame (for complete definitions, see \cite[Section III.2]{16}).
The braid monodromy of $S$ \cite{2} is a map $\varphi: \pi_1(E - N, u) \rightarrow B_p[D,K]$ defined as follows:
every loop in $E - N$ starting at $u$ has liftings to a system of $p$ paths in $(E - N) \times D$ starting at
each point of $K = 1,\ldots,p$. Projecting them to $D$ we obtain $p$ paths in $D$ defining a motion $\{1(t),
\dots , p(t)\}$ (for $0 \leq t \leq 1$) of $p$ points in $D$ starting and ending at $K$. This motion defines a
braid in $B_p [D , K]$. By the Artin Theorem \cite{17}, for $j=1, \dots, q$, there exists a half-twist $Z_j \in
B_p[D , K]$ and $\epsilon_j \in \Z$, \st $\varphi(\dl _j) = Z_j^{\epsilon_j}$, where  $Z_j$ is a half-twist and
$\epsilon_j = 1,2$ or $3$ (for an ordinary branch point, a node, or a cusp respectively). We now explain how to
describe the $Z_j$.

First, we recall the definition of an almost real curve from \cite{17}.

\begin{definition}
A curve $S$ is called an \emph{almost real curve} if
\begin{enumerate}
\item $N \subseteq E \cap \R$, \item $N \subset E - \partial E$, \item $\forall x \in {E \cap \R -N}, \# {(K
\cap \R)(x)} \ge p-1$, \item $\forall x \in N, \# {\pi^{-1}(x) \cap M'} = 1$, \item The singularities can be
\begin{enumerate}
\item a branch point, topologically locally equivalent to $y^2+x=0$ or $y^2-x=0$, \item a tacnode
(locally a line tangent to a conic) \item a local intersection of $m$ smooth branches, transversal to each other.
\end{enumerate}
\end{enumerate}
\end{definition}

We compute the braid monodromy around each singularity in $S$. Let $A_j$ be a singularity in $S$ and denote by
$x_j$ its projection by $\pi$ to the $x$-axis. We choose a point $x'_j$ next to $x_j$, \st \ $\pi^{-1}(x'_j)$ is
a typical fiber. If $A_j$ is (b), (c) or (d), then $x'_j$ is on the right side of $x_j$. If $A_j$ is (a), then
$x'_j$ is on the  left side of $x_j$ (the typical fiber in case (a), which is on the left side of this
singularity, intersects the conic in two real points). We encircle $A_j$ with a very small circle in such a way
that the typical fiber $\pi^{-1}(x'_j)$ intersects the circle in two points, say $a ,b$. We fix a skeleton
$\xi_{x'_j}$ which connects $a$ and $b$, and denote it as $<a,b>$. The Lefschetz diffeomorphism $\Psi$
\cite[Subsection 1.9.5]{3} allows us to get a resulting skeleton $(\xi_{x'_j}) \Psi$ in the typical fiber
$\C_u$. This one defines a motion of its two endpoints. This motion induces a half-twist $Z_j = \Delta <
(\xi_{x'_j}) \Psi>$. As above, $\varphi(\delta_j) = \Delta < (\xi_{x'_j}) \Psi>^{\epsilon_j}$. The braid
monodromy factorization associated to $S$ is $\Delta^2_{p} = \prodl^q_{j=1} \varphi(\delta_j)$. The degree of
the factorization (the sum of the $\epsilon_j$) is $p \times (p-1)$.

It is difficult to compute the braid monodromy factorization of a general branch curve directly, but it can be done indirectly by degeneration
\ref{Degeneration-Regeneration}.

\begin{figure}[h]\label{Degeneration-Regeneration}
\begin{minipage}{\textwidth}
\begin{center}
\epsfbox {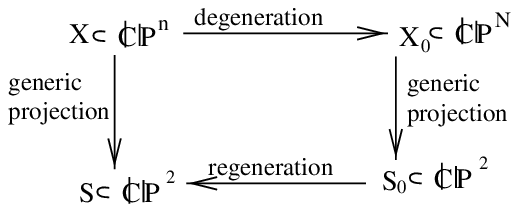}
\end{center}
\end{minipage}
\caption{}
\end{figure}
To do this, we
can degenerate the surface to a union of planes, whose branch curve is a line arrangement in the plane, whose
complement has a particularly simple fundamental group. The braid monodromy algorithm of \cite{17} computes the
braid monodromy factorization for the line arrangement. We can then regenerate the braid monodromy factorization
via the regeneration rules of \cite[p. 335-337]{19} to obtain the braid monodromy factorization of the branch
curve.

Note that geometrically, this regeneration of the braid monodromy factorization is equivalent to a \emph{local}
regeneration of the branch curve itself, via the regeneration lemmas in \cite[p.38]{17}:
\begin{enumerate}
\item regeneration of a branch point to two branch points; \item regeneration of a node to two or four nodes;
\item regeneration of a tangency point to three cusps.
\end{enumerate}

We note that a braid $Z^2_{i \; j}$ which is related to a node is regenerated to $Z^2_{i, j \;j'}$, which is a
product of two braids $Z^2_{i \; j}, Z^2_{i \;j'}$. A braid $Z^3_{i \; j}$ related to a cusp is regenerated to a
braid $Z^3_{i, j\;j'}$, which is  a product of three braids, as follows: $Z^3_{i\; j}$, $(Z^3_{i\; j})^{Z_{j\;
j'}}$, $(Z^3_{i \; j})^{Z^{-1}_{j \; j'}}$.

Throughout this paper we represent braids $B$ pictorially by paths $P$, such that $B$ is the half-twist with
respect to the path $P$, that is, consider a map $f$ from a closed neighborhood $V$ of $P$ to a disc $U$, taking
$P$ to the line segment $[-1,1];$ the braid $B$ is then defined by the conjugation of the 180 degree rotation in
$U$ by $f$ on $V,$ and the identity elsewhere. In the case of the products that arise from nodes and cusps, by
combinations of the individual paths representing the half twists along those paths. (For simplicity, in the
larger graphs, sometimes we draw one path representing all three in the portion where they overlap.  Any dotted
or dashed segment that connects to a path rather than a vertex is considered to continue along that path.)

After computing the braid monodromy factorization, the second step is to compute the fundamental group of the complement
of the branch curve. By the Van Kampen theorem \cite{vk}, there is a ``good" geometric base $\{\G_{j}\}$ of
$\pi_{1}(\C_{x_{0}} - S \cap \C_{x_{0}} , *),$ where $\C_{x_0}$ is the fiber of the projection
$\pi|_{\text{aff}}$ above $x_0$, such that the group $\pi_{1}(\C^2 - S, *)$ is generated by the images of $\{
\G_{j} \}$ in $\pi_{1}(\C^2 - S, *)$ with the relations $\varphi(\delta_{i}) \hspace{0.1cm} \G_{j} = \G_{j}
\hspace{0.1cm} \forall i,j$. Recall that
$$
\pi_{1}(\C \P^2 - \bar{S}) \simeq \pi_{1}(\C^2 - S) / \langle \prod_j{\G_j}\rangle.
$$

Recall that the branch curve $S$ of a smooth surface $X$ contains branch points of the projection (where the
curve is locally defined by the equation $y=x^2$), nodes (locally $y^2=x^2$) and cusps (locally $y^2 = x^3$). If
the surface has isolated nodal/cuspidal singularities, then the branch curve is of the same form. Denote by $a$
and $b$ the two branches of the real part of $S$ in a neighborhood of such a singular point, and let $\G_{a},
\G_{b}$ be two non-intersecting loops in $\pi_{1}(\C_{x_{0}} - S \cap \C_{x_{0}}, *)$ around the intersection
points of the branches with the fiber $\C_{x_{0}}$ (constructed by cutting each of the paths and creating two
loops, which proceed along the two parts and encircle $a$ and $b$); see \cite[Proposition-Example VI.1.1]{16}.
Then by the van Kampen Theorem, we have the relations $\langle \G_{a},\G_{b} \rangle = \G_{a} \G_{b} \G_{a}
\G_{b}^{-1} \G_{a}^{-1} \G_{b}^{-1} = 1$ for a cusp, $[\G_{a},\G_{b}] = \G_{a} \G_{b}\G_{a}^{-1} \G_{b}^{-1} =
1$ for a node, and $\G_a = \G_b$ for a branch point.

These relations generate the group completely in the affine case. In the projective case, we have the additional
"projective" relation $\prod \G_j =1.$

The third step is to use the theorem of Moishezon-Teicher, \cite{15}, that there is an exact sequence
\begin{equation} \label{M-T} 0 \rightarrow K \rightarrow \pi_1(\C\P^2 - \bar{S}) \rightarrow S_n \rightarrow 0,
\end{equation} where the second map takes the generators $\Gamma_i$ of the fundamental group to the transpositions in $S_n$ according to the lexicographic order, and the fundamental group $\pi_1(\Gal{X})$ is the quotient of $K$ by the relations $<\G_i^2>.$
We apply this exact sequence to obtain a presentation of the fundamental group of the Galois cover, and simplify
the relations to produce a canonical presentation and identify the group, using various new group theoretic
methods for each case.

\subsection{Contents}
The rest of this paper is organized as follows:  Section \ref{small} is a complete survey of the fundamental
groups of the Galois covers of all the surfaces of degrees less than $5$ that degenerate to a "nice" planar
arrangement, i.e., one in which no three planes meet in a line. In subsection \ref{Quadric}, we consider the unique smooth
quadric surface that degenerates to two planes. In subsection \ref{Cubics}, we consider the two possible cubics:
in \ref{F1}, we analyze the (1,1) embedding of the Hirzebruch surface $F_1$ as a smooth cubic. In \ref{Cayley},
we recall the degeneration of the Cayley cubic, which is singular, but note that there are also smooth surfaces
that share that degeneration, and analyze the fundamental group of its Galois cover.  In subsection
\ref{Quartics}, we consider the five possible quartics: in \ref{F2}, we recall the (1,1) embedding of the
Hirzebruch surface $F_2.$ In \ref{V4}, we recall the embedding of the Veronese surface $V_2$ and provide a more
explicit analysis of its fundamental groups, and in \ref{F0-12} we analyze the $(1,2)$ embedding of the
Hirzebruch surface $F_0,$ or $\C\P^1\times\C\P^1.$ In \ref{Cayley+} we analyze the new case of a quartic surface
degenerating to the union of a plane and the Cayley cubic degeneration, and in \ref{4-pt} we analyze the case of
a quartic degenerating to a $4$-point.

In the appendix, Section \ref{App}, we analyze for completeness those cases of $\C\P^1 \times \C\P^1$ not
covered by the Moishezon-Teicher theorem \cite[p.642]{15}, namely the $(2,2)$ embedding and the $(1,n)$
embedding for all positive integers $n.$

\section{Small Degree Cases}\label{small}

Throughout this paper, we use the standard triangulation representation in which triangles denote planes,
internal edges denote lines of intersection (we do not consider the boundary edges as they do not define well
defined lines) and points of intersection of the internal edges denote points of intersection of the
corresponding planes, and every arrangement is assumed to be embedded in the smallest dimensional projective
space.

It is clear that every plane arrangement can be represented by a triangulation as long as no three planes meet
in a line and no plane meets more than three other planes.
 In degrees less than $5,$ the second condition is vacuous.  For the triangulation to represent a meaningful planar degeneration, it
 is necessary that if exactly two triangles meet at a point, they must meet along an edge, and if three or more triangles meet at a point, the dual graph
 of the triangulation must be connected. The reason is that the branch curve
 of an irreducible surface (and hence of a smooth surface or a surface with only isolated non-removable singularities) is connected.
 Hence the branch curve of the degeneration must also be connected.
 Moreover, the triangulation is planar
 since we can consider the edge graph of the triangulation, and by Kuratowski's theorem, every non-planar graph contains
 a subgraph homeomorphic to the complete graph $K_5$ or the complete bipartite graph $K_{3,3},$
 neither of which is possible for triangulations of degree less than $10.$  Conversely, any
 triangulation determines a plane arrangement in a projective space of sufficiently high degree.

 We therefore consider planar triangulations such that the dual graph is connected.

The plane arrangement is not unique, but its combinatorial invariants, and therefore, in particular, its braid monodromy factorization and the fundamental group of the complement of the branch curve, are well defined.  It is therefore enough to consider each possible plane triangulation and to show that it actually corresponds to a degeneration of some smooth surface.

\subsection{Quadric}\label{Quadric}

The smooth quadric surface degenerates to two planes, see Figure \ref{quad-degeneration}.

\begin{figure}[ht]
\epsfxsize=1.5cm 
\begin{minipage}{\textwidth}
\begin{center}
\epsfbox{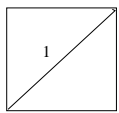}
\end{center}
\end{minipage}
\caption{Degeneration of the quadric}\label{quad-degeneration}
\end{figure}

The branch curve is a conic, which is smooth.  There is only one braid relation, which is the half twist given
in Figure \ref{quad-braid}.
\begin{figure}[ht]
\epsfxsize=1.5cm 
\begin{minipage}{\textwidth}
\begin{center}
\epsfbox{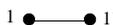}
\end{center}
\end{minipage}
\caption{The half-twist $Z_{1 \; 1'}$}\label{quad-braid}
\end{figure}
 The fundamental group of the complement has two generators $\G_1$ and
$\G_1'.$ The projective relation is $\G_1\G_1'=e$.  Hence $\G_1^2=e$. This group is thus $S_2,$ so the map
$\rho$ in Exact Sequence (\ref{M-T}) is the identity. Its kernel is therefore trivial.  Note that every degree
$2$ cover is Galois.

\subsection{Cubics}\label{Cubics}
\begin{theorem}
There are two possible "nice" cubic degenerations (i.e., degenerations to line arrangements such that no three
planes meet in a line), those shown in Figures \ref{trap3} and \ref{ozen-haman}.
\end{theorem}
\begin{proof}
We construct the plane arrangements by gluing triangles together.  The simplest connected arrangement is that
shown in Figure \ref{trap3}.  By one extra gluing we obtain Figure \ref{ozen-haman}.  It is not possible to glue
the triangles any further without either gluing three triangles in one edge or identifying the two endpoints of
some edge, both of which are forbidden.
\end{proof}

\subsubsection{The Hirzebruch surface $F_1$}\label{F1}
 The Hirzebruch surface $F_1,$ the ruled surface defined by $\E=\O \oplus \O(-1)$ on $\mathbb{P}^1,$
is embedded as a smooth cubic in $\mathbb{P}^4$ by the (1,1) embedding, i.e., by the linear system $|\ell_1 +
\ell_2|,$ where $\ell_1$ is the $(-1)$-curve and $\ell_2$ is the fibre. This surface degenerates to a union of
three planes, as depicted in Figure \ref{trap3}. The branch curve $C_0$ is a line arrangement consisting of two
intersecting lines. Regenerating it, we obtain the conic (1, 1') and the tangent line (2). When the line
regenerates, the tangency regenerates into three cusps.
\begin{figure}[h]
\epsfysize=2cm 
\begin{minipage}{\textwidth}
\begin{center}
\epsfbox{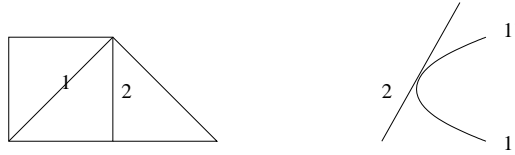}
\end{center}
\end{minipage}
\caption{Degeneration and Regeneration of the Hirzebruch surface $F_1$}\label{trap3}
\end{figure}

The braid monodromy factors are thus given in Figure \ref{trap3-braids}.  Here and throughout this paper, we use the group theoretic convention of denoting by $a^b$ the conjugation $b^{-1}ab.$
\begin{figure}[ht]
\epsfxsize=3cm 
\begin{minipage}{\textwidth}
\begin{center}
\epsfbox{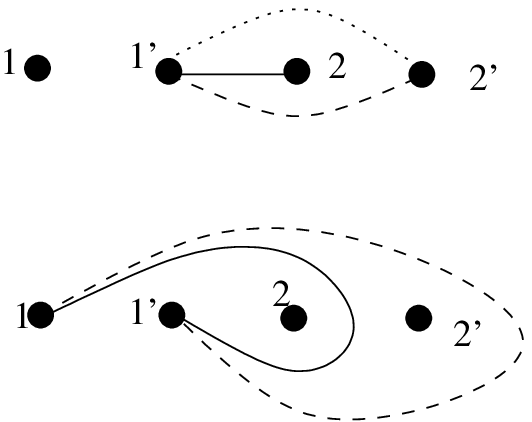}
\end{center}
\end{minipage}
\caption{$Z^3_{1', 2\;2'}$, $\left(Z_{1\;1'}\right)^{Z^2_{1', 2\;2'}}$}\label{trap3-braids}
\end{figure}

\begin{theorem}
The fundamental group of the Galois cover of the Hirzebruch surface $F_1$ is trivial.
\end{theorem}
\begin{proof}
By the Van Kampen theorem, the fundamental group $\pi_1(\C^2 \setminus S)$ is generated by the four generators
$\G_1$, $\G_1',$ $\G_2$ and $\G_2'$, subject to the relations
\begin{tiny}
\begin{eqnarray}
{}\G_1 &=& \G_1',\\
{}\G_2 &=& \G_2',\\
{}\langle\G_{1}, \G_{2}\rangle & = & e, \label{Hfin1}\\
{}\G_1^{-2} & = & \G_2^2. \label{Hfin2}
\end{eqnarray}
\end{tiny}

The map $\pi_1(\C^2 \setminus S) \rightarrow S_3$ is given by \ \ $\G_1  \mapsto  s_1$ \ and \ $\G_2  \mapsto
s_2$. Thus  the kernel $K$ is generated by $\G_1^2$ and $\G_2^2.$  Hence, the fundamental group $\pi_1(\Gal{X})$
is trivial.
\end{proof}

\subsubsection{Triple point}\label{Cayley} The union of three planes meeting at
a triple point, as shown in Figure \ref{ozen-haman}, was first studied as a degeneration of the Cayley cubic, a
singular surface with four nodes. However, there are also smooth cubic surfaces that degenerate to this union:
for example, the surface in $\C\P^3$ defined by $xyz+tw(x^2+y^2+z^2+w^2),$ which degenerates when $t=0$ to
$xyz.$

\begin{figure}[ht]
\epsfysize=2cm 
\begin{minipage}{\textwidth}
\begin{center}
\epsfbox{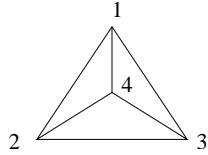}
\end{center}
\end{minipage}
\caption{The degeneration of the Cayley cubic}\label{ozen-haman}
\end{figure}

In the paper \cite{Cayley}, the initial braid monodromy factorization of the degenerate surface consisting of
three planes meeting at a point is found to be $\Delta^2_3$, and by regenerating, the following braid monodromy
factorization is obtained for the complement of the branch curve $S$:
\begin{proposition}\cite{Cayley}
The braid monodromy factorization of $S$ is given in (\ref{del1}), and its factors are represented by the paths
in Figure \ref{mondel1}.
\begin{eqnarray} \label{del1}
{\Delta^2_S} = & {(Z_{1' \; 3}^2)}^{Z^{-2}_{2', 3 \; 3'}} \cdot (Z_{1' \; 3'}^2)^{Z_{1' \; 3}^2 Z^{-2}_{2', 3 \;
3'}} \cdot {(Z_{1 \; 3}^2)}^{Z^{-2}_{2', 3 \; 3'}} \cdot (Z_{1 \; 3'}^2)^{Z_{1 \; 3}^2 Z^{-2}_{2', 3 \; 3'}} \\
& \cdot {(Z_{2 \; 2'})}^{Z^{-2}_{1 \; 1', 2} \bar{Z}^2_{2, 3 \; 3'}} \cdot {(Z^3_{2, 3 \; 3'})}^ {Z^2_{2 \; 2'}}
\cdot  Z^3_{1 \; 1', 2'} \cdot Z_{2 \; 2'} \cdot Z_{1,1'} \cdot Z_{3,3'}. \nonumber
\end{eqnarray}

Note that the first, the fourth and the last two paths correspond to braids of branch points. The second and
third paths correspond to braids of cusps, and the rest correspond to braids of nodes.

\begin{figure}[ht]
\epsfysize=11cm 
\begin{center}
\epsfbox{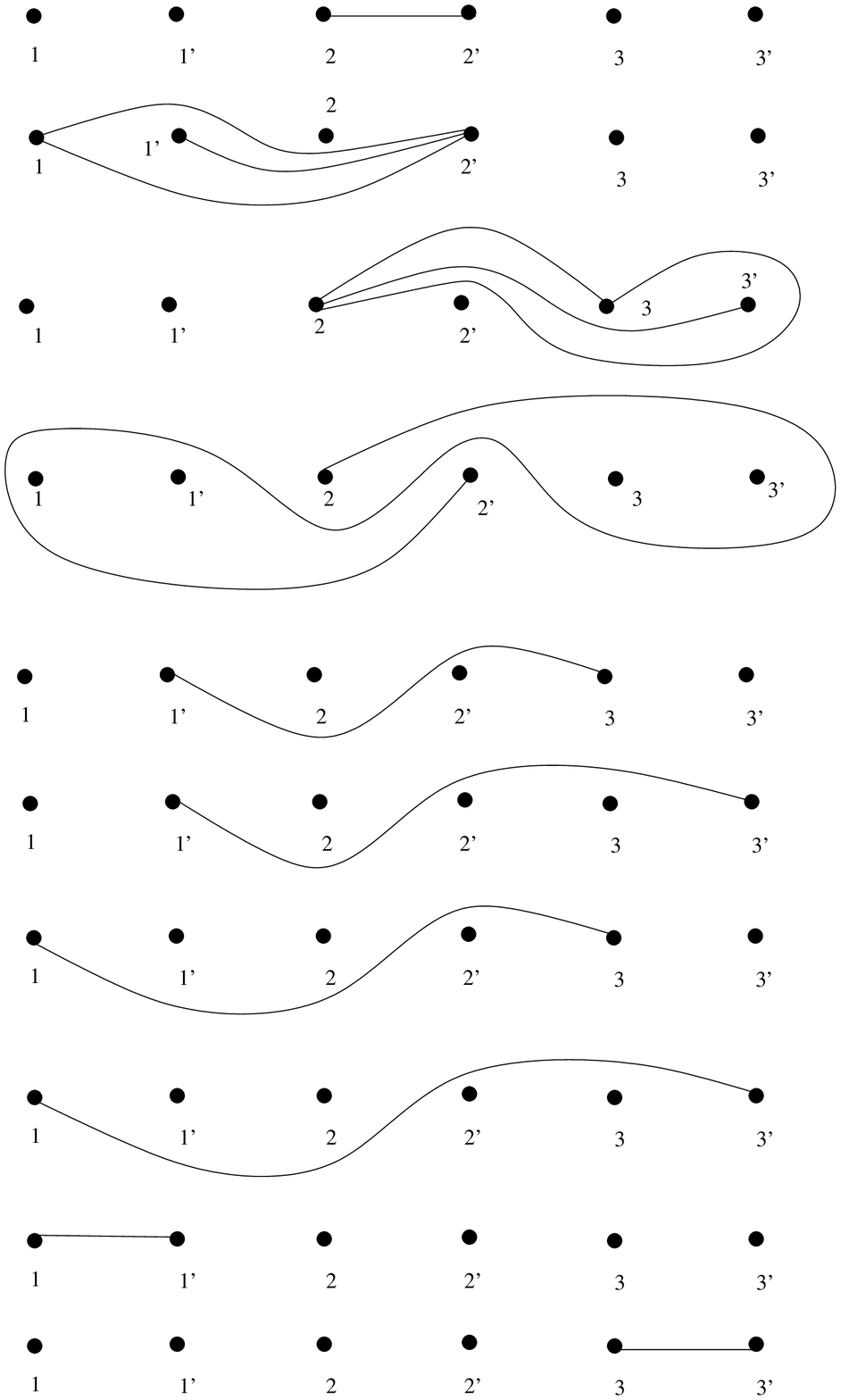}
\end{center}
\caption{Braid monodromy factorization of the Cayley branch curve}\label{mondel1}
\end{figure}
\end{proposition}

\begin{proposition}\cite{Cayley}
The fundamental group $\pi_1(\C^2 - S)$ is generated by $\G_1, \G_{2}, \G_{3},$ subject to the relations
\begin{tiny}
\begin{eqnarray}
{}\langle\G_{1}, \G_{2}\rangle & = & e, \label{fin1}\\
{}\langle\G_{2}, \G_{3}\rangle & = & e, \label{fin2}\\
{}[ \G_{2}, \G_{1}^2 \G_{3}^2] & = & e, \label{fin3}\\
{}[ \G_{1}, \G_{2}^{-1} \G_{3} \G_2] & = & e,\label{fin4}
\end{eqnarray}
\end{tiny}
where we denote by $\langle a,b\rangle=e$ the braid relation between $a$ and $b,$ $aba=bab.$  The group
$G=\pi_1(\C\P^2 - \bar{S})$ has the relations (\ref{fin1}), (\ref{fin2}), (\ref{fin4}) and the additional
projective relation
\begin{tiny}
\begin{eqnarray}\label{relpro}
\G_{3}^2 \G_{2}^2 \G_{1}^2 & = & e.
\end{eqnarray}
\end{tiny}
\end{proposition}

\begin{theorem}
The fundamental group of the Galois cover of the Cayley cubic (or its smoothing) is $\mathbb{Z}_2 \times
\mathbb{Z}_2$.
\end{theorem}

\begin{proof}
We relabel the generators of $G$ by replacing $\G_3$ by $\G_3' = \G_2^{-1}\G_3 \G_2,$ to obtain the relations
(\ref{fin1}) and \begin{tiny}\begin{eqnarray}
{}\langle\G_2, \G_3'\rangle &=& e\\
{}[\G_1, \G_3'] &=&e\\
{}\G_2\G_3'^2\G_2\G_1^2 &=&e.\label{8lab}
\end{eqnarray}
\end{tiny}
Thus, this group is a quotient of the braid group $B_4$ by the relation (\ref{8lab}). Since this relation is in
the kernel of the  map $B_4 \rightarrow S_4,$ our fundamental group has a natural map to $S_4,$ whose
kernel is normally generated by $\G_1^2,$ $\G_2^2,$ and $\G_3'^2.$  So $G/\langle \G_1^2, \G_2^2, \G_3^2\rangle \cong S_4.$  Composing this map with a surjective map $S_4
\rightarrow S_3,$ whose kernel is isomorphic to $\mathbb{Z}_2 \times \mathbb{Z}_2$ and is normally generated by
$s_1s_3,$ we obtain the map $\pi_1(\C\P^2 - \bar{S}) \rightarrow S_3.$  The kernel $K/\langle \G_1^2, \G_2^2, \G_3^2\rangle $ is isomorphic to $\mathbb{Z}_2 \times \mathbb{Z}_2,$ generated by
$\G_1\G_3'$ and its conjugates.  Hence we obtain $\pi_1(\Gal{X}) = K / \langle
\G_1^2, \G_2^2, \G_3'^2\rangle \cong \mathbb{Z}_2 \times \mathbb{Z}_2.$\end{proof}

\subsection{Quartics}\label{Quartics}
\begin{theorem}
There are five possible quartic degenerations, corresponding to Figures \ref{trapez4}, \ref{Ver4}, \ref{(1,2)},
\ref{ozen-haman+} and \ref{4-ptsurf}.
\end{theorem}
\begin{proof}
We construct the degenerations combinatorially by gluing triangles.  Beginning with the arrangement of
triangles found in Figure \ref{trap3}, we can add one more triangle in different places to obtain Figures
\ref{trapez4}, \ref{Ver4} and \ref{(1,2)}.  Beginning with that found in Figure \ref{ozen-haman}, we can add one
more triangle to obtain the quartic arrangement in Figure \ref{ozen-haman+}, and then glue two more edges
together to obtain that in Figure \ref{4-ptsurf}.

It is not possible to glue them any further (for instance, to obtain a non-simply-connected arrangement like the
torus degenerations that appear in \cite{fr2} in higher degrees), because in degree $4$ this would force us to
have either three planes meeting in a line, or two lines through the same pair of points, both of which are
forbidden configurations.
\end{proof}

\subsubsection{The surface $F_2$}\label{F2} Consider the Hirzebruch surface $F_2$ defined by $\E=\O \oplus \O(-2)$
on $\mathbb{P}^1.$  It is a toric variety and can be embedded in $\C\P^5$ by the linear system $|\ell_1 +
\ell_2|,$ where $\ell_1$ is the $(-2)$-section and $\ell_2$ is the fibre.  Its degeneration is a union of four
planes in $\C\P^5$, as depicted in Figure \ref{trapez4}, as shown in \cite{Ogata}.

\begin{figure}[ht!]
\epsfysize=2cm 
\begin{minipage}{\textwidth}
\begin{center}
\epsfbox{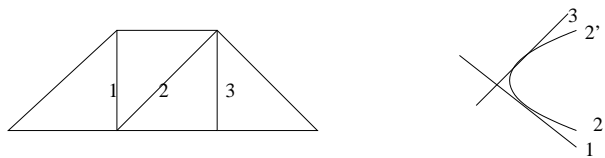}
\end{center}
\end{minipage}
\caption{Degeneration and Regeneration of the Hirzebruch surface $F_2$}\label{trapez4}
\end{figure}

The branch curve $S_0$ in $\C\P^2$ is an arrangement of three lines. Regenerating it, the diagonal line
regenerates to a conic, which is tangent to the lines $1$ and $3$. When the lines regenerate, each tangency
regenerates into three cusps. We obtain the branch curve $S$, which is of degree $6$ and which has six cusps.

As computed in \cite{Ogata}, the braid monodromy factorization corresponding to $C$ gives an expression of
$\Delta_6^2$ as the braids shown in Figure \ref{F_2-braids},

\begin{figure}[h]
\epsfxsize=3cm 
\begin{minipage}{\textwidth}
\begin{center}
\epsfbox{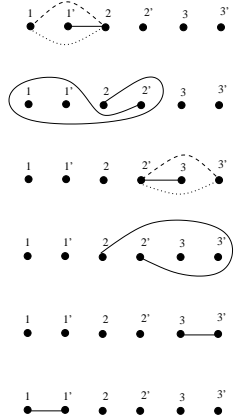}
\end{center}
\end{minipage}
\caption{$Z^3_{1 \; 1', 2} \cdot {Z_{2\; 2' }}^{Z^2_{1\;1', 2}}, Z^3_{2', 3\;3'} \cdot {Z_{2\; 2'}}^{Z^2_{2',
3\; 3'}}, Z_{3\; 3'}, Z_{1 \; 1'}$}\label{F_2-braids}
\end{figure}

and we also have the parasitic intersection braid given in Figure \ref{F_2-para}.

\begin{figure}[h]
\epsfxsize=3cm 
\begin{minipage}{\textwidth}
\begin{center}
\epsfbox{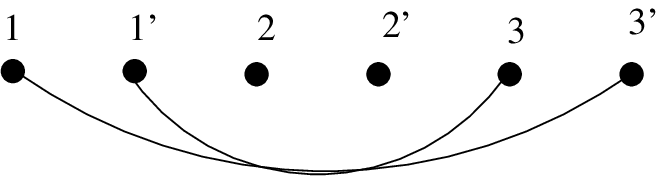}
\end{center}
\end{minipage}
\caption{${Z^2}_{\hspace{-.1cm}{1 \; 1',3 \; 3'}}$}\label{F_2-para}
\end{figure}

We apply the van Kampen Theorem to the above braids to get a presentation for $\pi_1(\C^2 \setminus S)$.

After simplifying the relations, we find that the fundamental group is generated by $\G_{1}, \G_{2}, \G_{3}$
subject to the relations \begin{tiny}\begin{eqnarray*}
{}\langle\G_{i} , \G_{i+1}\rangle & = & e, \ \ \mbox{\emph{for i=1, 2,}} \\
{}[ \G_{1} , \G_{3} ] & = & e,\\
{}\G_1^{-2} \G_{2} \G_1^{2} & = & \G_3^{-2} \G_{2} \G_3^{2}.
\end{eqnarray*}
\end{tiny}
Clearly, the group $\pi_1(\C\P^2 \setminus S)$ is isomorphic to the braid group quotient ${\mathcal{B}_4} /
{\langle\G_2 \G_3^2 \G_2 \G_1^2\rangle}$.  The map to $S_4$ is given by $\G_i \mapsto s_i,$ and its kernel $K$
is normally generated by the $\G_i^2.$  Hence, the quotient $\pi_1(\Gal{X})=K/\langle\G_i^2\rangle$ is trivial.

\subsubsection{The Veronese surface $V_2$}\label{V4} We consider the Veronese surface of order $2$, i.e., the
embedding of $\mathbb{CP}^2$ into $\mathbb{CP}^5$ given by $(x:y:z) \mapsto (x^2:y^2:z^2:xy:yz:xz).$ This is a
surface of degree $4,$ which was treated in a very abstract manner by \cite{GalCovs}.  The fundamental group of
the Galois cover was found there to be $\mathbb{Z}^4.$

The degeneration is not necessary in this case, but for the sake of completeness in our survey of plane
arrangements, we note that the surface degenerates to the union of four planes depicted in Figure \ref{Ver4},
and we present the fundamental group explicitly in terms of generators and relations.

\begin{figure}[h]
\epsfysize=2.5cm 
\begin{minipage}{\textwidth}
\begin{center}
\epsfbox{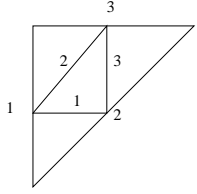}
\end{center}
\end{minipage}
\caption{Degeneration of Veronese}\label{Ver4}
\end{figure}

The branch curve consists of three lines meeting at three different vertices. We regenerate each vertex in turn,
and use the van Kampen Theorem to obtain relations among the generators $\G_i$ and $\G_i'$ (for $i=1,$ $2,$ $3$)
of the fundamental group $\pi_1(\C\P^2 \setminus S)$ of the complement of the branch curve in the projective
plane.

Vertex 1 regenerates to a line (1) tangent to a conic (2,2'), as in Figure \ref{V1fig}.
\begin{figure}[h]
\epsfysize=1.5cm 
\begin{minipage}{\textwidth}
\begin{center}
\epsfbox{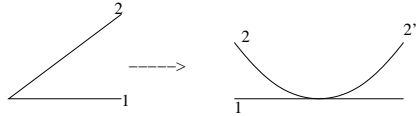}
\end{center}
\end{minipage}
\caption{Regeneration of Vertex 1 of Veronese}\label{V1fig}
\end{figure}

It gives rise to the braid monodromy factors $Z^3_{1\;1', 2}$ and $\left(Z_{2\;2'}\right)^{Z_{1\;1', 2}^2},$
which by the Van Kampen theorem yield the relations \begin{tiny} \begin{eqnarray}
{}\langle\G_1 , \G_2\rangle & = & e, \label{V1}\\
{}\langle\G_1' , \G_2\rangle & = & e, \label{V1'}\\
{}\langle\G_1^{-1}\G_1'\G_1, \G_2\rangle & = & e, \label{V1conj}\\
{}\G_2\G_1'\G_1\G_2\G_1^{-1}\G_1'^{-1}\G_2^{-1} &=& \G_2'.\label{V1a}
\end{eqnarray}\end{tiny}

Vertex 2 regenerates to a line (1) tangent to a conic (3,3'), and Vertex 3 to a line (3) tangent to a conic
(2,2'). \forget , as in Figure \ref{V2fig}.
\begin{figure}[h]
\begin{minipage}{\textwidth}
\begin{center}
\epsfbox{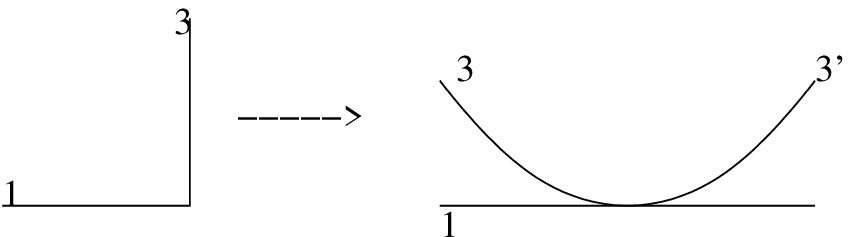}
\end{center}
\end{minipage}
\caption{Regeneration of Vertex 2 of Veronese}\label{V2fig}
\end{figure}
\forgotten They give rise to the braid monodromy factors $Z^3_{1\;1', 3}$ and
$\left(Z_{3\;3'}\right)^{Z^2_{1\;1', 3}},$ which yield the relations \begin{tiny}\begin{eqnarray}
{}\langle\G_1 , \G_3\rangle & = & e, \label{V2}\\
{}\langle\G_1' , \G_3\rangle & = & e, \label{V2'}\\
{}\langle\G_1^{-1}\G_1'\G_1, \G_3\rangle & = & e, \label{V2conj}\\
{}\G_3\G_1'\G_1\G_3\G_1^{-1}\G_1'^{-1}\G_3^{-1} &=& \G_3'.\label{V2a},
\end{eqnarray}\end{tiny}
\forget
 Vertex 3 regenerates to a line (3) tangent to a conic (2,2'). as in Figure \ref{V3fig}.
\begin{figure}[h]
\begin{minipage}{\textwidth}
\begin{center}
\epsfbox{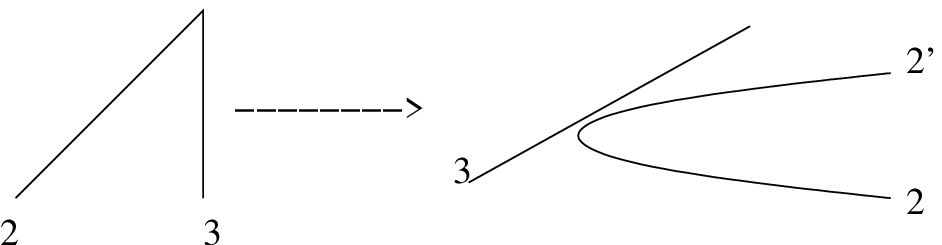}
\end{center}
\end{minipage}
\caption{Regeneration of Vertex 3 of Veronese}\label{V3fig}
\end{figure}
\forgotten
 and to the braid monodromy factors $Z^3_{2', 3\;3'}$
and $\left(Z_{2\;2'}\right)^{Z^2_{2', 3\;3'}},$ which yield the relations \begin{tiny}\begin{eqnarray}
{}\langle\G_2' , \G_3\rangle & = & e, \label{V3}\\
{}\langle\G_2', \G_3'\rangle & = & e, \label{V3'}\\
{}\langle\G_2', \G_3^{-1}\G_3'\G_3\rangle & = & e, \label{V3conj}\\
{}\G_3'\G_3\G_2'\G_3^{-1}\G_3'^{-1} &=& \G_2\label{V3a}.
\end{eqnarray}\end{tiny}

 In addition to these relations, we have the projective relation
\begin{tiny}\begin{equation}\label{Vproj} \G_3'\G_3\G_2'\G_2\G_1'\G_1 = e.
\end{equation}\end{tiny}

\forget
 From (\ref{V1conj}),
$$\langle\G_1', \G_1\G_2\G_1^{-1}\rangle=e.$$
Using (\ref{V1}), we get that $$\langle\G_1', \G_2^{-1}\G_1\G_2\rangle=e.$$ Hence
$$\langle\G_2\G_1'\G_2^{-1},\G_1\rangle=e.$$ Hence, using (\ref{V1'}),
$$\langle\G_1'^{-1}\G_2\G_1', \G_1\rangle=e.$$ Hence we get
\begin{equation}\label{V13}
\langle\G_2, \G_1'\G_1\G_1'^{-1}\rangle=e.
\end{equation}

>From (\ref{V1a}), we get
$$\G_2\G_1'(\G_1\G_2\G_1^{-1})\G_1'^{-1}\G_2^{-1}=\G_2'.$$
Using (\ref{V1}), we get
$$\G_2\G_1'\G_2^{-1}\G_1\G_2\G_1'^{-1}\G_2^{-1}=\G_2'.$$
Using (\ref{V1'}), this becomes
$$\G_1'^{-1}\G_2(\G_1'\G_1\G_1'^{-1})\G_2^{-1}\G_1' = \G_2'.$$
By (\ref{V13}),
$$\G_1'^{-1}\G_1'\G_1^{-1}\G_1'^{-1}\G_2\G_1'\G_1\G_1'^{-1}\G_1'=\G_2',$$
which reduces to
$$\G_1^{-1}\G_1'^{-1}\G_2\G_1'\G_1 = \G_2'.$$

Hence, we have
\begin{equation}\label{V14}
\G_2=\G_1'\G_1\G_2'\G_1^{-1}\G_1'^{-1},
\end{equation}
 or equivalently,
 $$\G_2'=\G_1^{-1}\G_1'^{_1}\G_2\G_1'\G_1.$$

Similarly, using the equations (\ref{V2})-(\ref{V2a}) in place of (\ref{V1})-(\ref{V1a}), and substituting
$\G_3$ and $\G_3'$ for $\G_2$ and $\G_2',$ respectively, we have
\begin{equation}\label{V15}
\langle\G_3, \G_1'\G_1\G_1'^{-1}\rangle=e
\end{equation}
and
\begin{equation}\label{V16}
\G_3=\G_1'\G_1\G_3'\G_1^{-1}\G_1'^{-1},
\end{equation}
or equivalently,
$$\G_3'=\G_1^{-1}\G_1'^{-1}\G_3\G_1\G_1'.$$

By (\ref{V14}), $$\langle\G_1,\G_2'\rangle=\langle\G_1, \G_1^{-1}\G_1'^{-1}\G_2\G_1'\G_1\rangle=$$
$$=\langle\G_1, \G_1'^{-1}\G_2\G_1'\rangle = \langle\G_1'\G_1\G_1'^{-1}, \G_2\rangle = e.$$

Hence, we get
\begin{equation}\label{V17}
\langle\G_1, \G_2'\rangle =e.
\end{equation}

By (\ref{V14}), $$\langle\G_1', \G_2'\rangle = \langle \G_1', \G_1^{-1}\G_1'^{-1}\G_2\G_1'\G_1\rangle =
\langle\G_1', \G_1'\G_1^{-1}\G_1'^{-1}\G_2\G_1'\G_1\G_1'^{-1}\rangle,$$ which by (\ref{V13}) reduces to
$$\langle\G_1', \G_2\G_1'\G_1\G_1'^{-1}\G_2^{-1}\rangle = \langle\G_2^{-1}\G_1'\G_2, \G_1'\G_1\G_1'^{-1}\rangle.$$ By
\ref{V1'}, we can convert this to $$\langle\G_1'\G_2\G_1'^{-1}, \G_1'\G_1\G_1'^{-1}\rangle = \langle\G_2,
\G_1\rangle =e,$$ where the second equality is due to (\ref{V1}).

Hence, we have obtained
\begin{equation}\label{V18}
\langle\G_1', \G_2'\rangle =e.
\end{equation}

In the same way, using (\ref{V2})-(\ref{V3a}) and (\ref{V15})-(\ref{V16}), we get the following relations:
\begin{eqnarray}
\langle\G_1, \G_3'\rangle&=& e \label{V19} \\
\langle\G_1', \G_3'\rangle &=& e \label{V20} \\
\langle\G_2, \G_3\rangle &=& e \label{V21} \\
\langle\G_2, \G_3'\rangle &=& e \label{V22}
\end{eqnarray}

In other words, $$\langle\G_i, \G_j\rangle = \langle\G_i', \G_j\rangle = \langle\G_i', \G_j'\rangle =e,$$ where
$i \neq j,$ $1 \leq i \leq 3,$ and $1 \leq j \leq 3.$

Using (\ref{V3a}), we get $$\G_3'\G_3\G_2'\G_3^{-1}\G_3'^{-1}=\G_2,$$ so
$$\G_3\G_2'\G_3^{-1}=\G_3'^{-1}\G_2\G_3'.$$ From (\ref{V3}) and (\ref{22}), it follows that $$\G_2'^{-1}\G_3\G_2'=
\G_2\G_3'\G_2^{-1};$$ hence $$\G_3=\G_2'\G_2\G_3'\G_2^{-1}\G_2'^{-1},$$ or equivalently,
$$\G_3'=\G_2^{-1}\G_2'^{-1}\G_3\G_2'\G_2.$$

Similarly, from (\ref{V14}) and (\ref{V16}), we have
\begin{eqnarray}
\G_1 &=& \G_2'\G_2\G_1'\G_2^{-1}\G_2'^{-1},\\
\G_1 &=& \G_3'\G_3\G_1'\G_3^{-1}\G_3'^{-1}, \\
\G_1' &=& \G_2^{-1}\G_2'^{-1}\G_1\G_2'\G_2,\\
\G_1' &=& \G_3^{-1}\G_3'^{-1}\G_1\G_3'\G_3;\\
\end{eqnarray}

in other words, $$\G_i'\G_i\G_j'\G_i^{-1}\G_i'^{-1}=\G_j$$ and $$\G_i^{-1}\G_i'^{-1}\G_j\G_i'\G_i= \G_j',$$ for
$i \neq j,$ $1 \leq i \leq 3,$ and $1 \leq j \leq 3.$

Since $\G_2'$ and $\G_3'$ are conjugates of $\G_2$ and $\G_3,$ \forgotten
 After simplifications, we find that
the group $G$ is generated by the four generators $\G_1,$ $\G_2,$ $\G_3$ and $\G_1',$ under the relations
$$\langle\G_1, \G_2\rangle =\langle\G_2, \G_3\rangle = \langle\G_3, \G_1\rangle=e  \ \  \mbox{and} \ \
\langle\G_1', \G_2\rangle =\langle\G_3, \G_1'\rangle =e,$$ and the map $\rho$ maps $\G_1 \mapsto (1,2),$ $\G_2
\mapsto (2,3),$ $\G_3 \mapsto (2,4)$ and $\G_1' \mapsto (1,2).$

\begin{lemma}\label{*lem}
$\Ker \rho$ is generated modulo the $\G_i^2$ by
$$(\G_1\G_2\G_3\G_2)^2, \ (\G_2\G_3\G_1\G_3)^2,  \ (\G_1'\G_2\G_3\G_2)^2, \ \mbox{and} \ (\G_2\G_3\G_1'\G_3)^2.$$
\end{lemma}

\noindent \textbf{Proof of Lemma \ref{*lem}.}  Consider the groups $$G_1=\left\langle\G_1, \G_2, \G_3 |
\langle\G_1, \G_2\rangle, \langle\G_2, \G_3\rangle, \langle\G_3, \G_1\rangle, \G_i^2\right\rangle$$
$$\mbox{and}\ \ \ G_2=\left\langle\G_1', \G_2, \G_3 | \langle\G_1', \G_2\rangle, \langle\G_2, \G_3\rangle, \langle\G_3,
\G_1'\rangle, \G_i^2, \G_1'^2\right\rangle.$$

The maps $\rho$ from $G_1$ and $G_2$ to $S_4$ are given by the following two diagrams.  As in \cite{RTV}, we represent maps from Coxeter groups to $S_n$ by diagrams with $n$ vertices, such that two vertices $i$ and $j$ are connected
by an edge labelled $\Gamma$ if $\Gamma$ maps onto the transposition $(i,j).$  Note that if the edges $\Gamma_k$
and $\Gamma_l$ meet in a vertex, then $\langle \Gamma_k, \Gamma_l\rangle=e,$ and if $\G_k$ and $\G_l$ are
disjoint, then $[\G_k,\G_l]=e,$ since if $\G_k\mapsto (a,b)$ and $\G_l\mapsto (b,c),$ then $\langle (a,b),
(b,c)\rangle=e,$ and if $\G_k\mapsto (a,b)$ and $\G_l\mapsto (c,d),$ then $[(a,b),(c,d)]=e.$  The lemma thus
follows from \cite[Theorem 2.3, p. 4]{RTV}. \hspace{1cm} $\Box$

\begin{figure}[ht]
\epsfxsize=3cm 
\begin{minipage}{\textwidth}
\begin{center}
\epsfbox{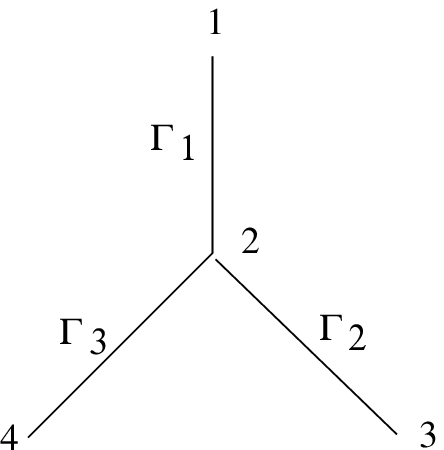}
\end{center}
\end{minipage}
\caption{}\label{Vmap1}
\end{figure}

\begin{figure}[ht]
\epsfxsize=3cm 
\begin{minipage}{\textwidth}
\begin{center}
\epsfbox{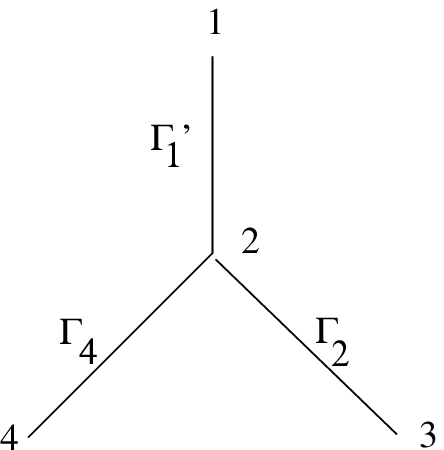}
\end{center}
\end{minipage}
\caption{}\label{Vmap2}
\end{figure}

Using the lemma, since the four generators commute modulo the $\G_i^2,$ it follows that $\Ker \rho
/\langle\G_i^2\rangle = \mathbb{Z}^4.$

\begin{remark}
As the branch curve is in fact a curve of degree 6 with nine cusps, the fundamental group of the complement of
the curve can also be deduced from Zariski \cite{zar}.
\end{remark}

\subsubsection{The embedding of $\C\P^1 \times \C\P^1$ as a quartic}\label{F0-12} Consider the degeneration of the
surface $\C\P^1 \times \C\P^1,$ embedded in $\mathbb{P}^5$ by the linear system $|\ell_1 + 2 \ell_2|,$ where
$\ell_1$ and $\ell_2$ are the pullbacks of the point classes from the two factors of $\C\P^1 \times \C\P^1$. The
degeneration is to a union of four planes depicted in Figure \ref{(1,2)}.
\begin{figure}[ht]
\epsfysize=2cm 
\begin{minipage}{\textwidth}
\begin{center}
\epsfbox{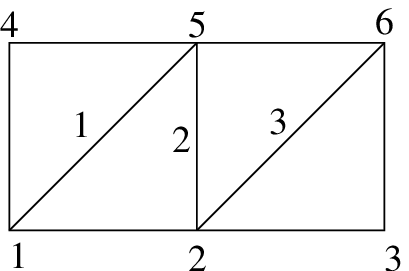}
\end{center}
\end{minipage}
\caption{The $(1,2)$-degeneration of $\C\P^1 \times \C\P^1$}\label{(1,2)}
\end{figure}

\begin{theorem}
The fundamental group of the Galois cover of $\C\P^1 \times \C\P^1$ in the $(1,2)$-embedding is trivial.
\end{theorem}

\begin{proof}
The branch curve of the degeneration consists of three lines meeting at two different vertices.

Vertex $1$ (respectively, $6$) regenerates to a conic and gives rise to the trivial braid $Z_{1,1'}$ (resp.
$Z_{6,6'}$), and hence to the relations
\begin{tiny}
\begin{eqnarray}
{}\G_1=\G_1',\label{12a}\\
{}\G_3=\G_3'.\label{12b}
\end{eqnarray}
\end{tiny}

Vertex $5$ (resp. $2$) regenerates to a line $2$ tangent to a conic (1,1') (resp., (3,3'), as in Figure
\ref{(12v5)} (resp., \ref{(12v2)}), giving rise to the braid monodromy factors $Z^3_{1', 2\;2'},$
$\left(Z_{1\;1'}\right)^{Z^2_{2\;2'}}$ and $\left(Z_{1\;1'}\right)^{Z^2_{1', 2\;2'}},$ (resp., $Z^3_{2\;2', 3},$
$\left(Z^2_{3\;3'}\right)^{Z_{2\;3}^2},$ and $\left(Z_{3\;3'}\right)^{Z_{2\;2', 3}^2}$), and the relations
\begin{tiny}\begin{eqnarray}
{}&&\langle\G_1', \G_2\rangle = \langle\G_1', \G_2'\rangle = \langle\G_1', \G_2^{-1}\G_2'\G_2\rangle = e,\label{12c}\\
{}&&\G_1=\G_2'\G_2\G_1'\G_2^{-1}\G_2'^{-1}, \label{12e}\\
{}&&\langle\G_2, \G_3\rangle = \langle\G_2', \G_3\rangle = \langle\G_2^{-1}\G_2'\G_2, \G_3\rangle = e,\label{12d}\\
{}&&\G_3'=\G_3\G_2'\G_2\G_3\G_2^{-1}\G_2'^{-1}\G_3^{-1}.\label{12f}
\end{eqnarray}\end{tiny}

\begin{figure}[ht]
\epsfysize=1.7cm 
\begin{minipage}{\textwidth}
\begin{center}
\epsfbox{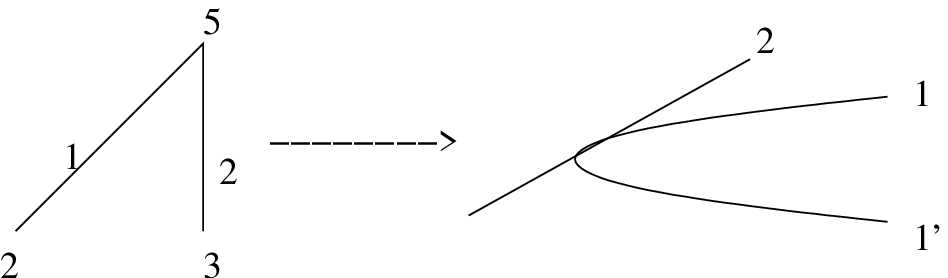}
\end{center}
\end{minipage}
\caption{Regeneration of Vertex 5}\label{(12v5)}
\end{figure}

\begin{figure}[ht]
\epsfysize=2cm 
\begin{minipage}{\textwidth}
\begin{center}
\epsfbox{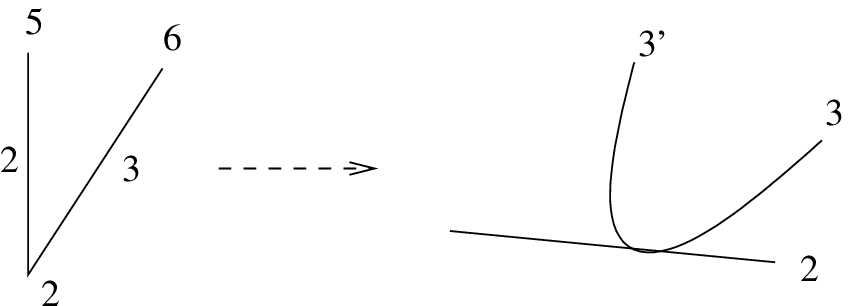}
\end{center}
\end{minipage}
\caption{Regeneration of Vertex 2}\label{(12v2)}
\end{figure}

We have also intersections that arise from lines that did not meet in the plane arrangement, but their images
meet in $\mathbb{CP}^2$. We call these parasitic intersections. Explanation and configuration of how to
construct the braids that correspond to these intersections, appear in \cite[p. 616]{15}.

They give rise to the following relations:
\begin{tiny}\begin{eqnarray}
{}[\G_2'\G_2\G_1'\G_2^{-1}\G_2'^{-1}, \G_3] &=& e, \label{12g}\\
{}[\G_2'\G_2\G_1'\G_2^{-1}\G_2'^{-1}, \G_3^{-1}\G_3'\G_3] &=& e, \label{12h}\\
{}[\G_2'\G_2\G_1'\G_1\G_1'^{-1}\G_2^{-1}\G_2'^{-1}, \G_3] &=& e, \label{12i}\\
{}[\G_2'\G_2\G_1'\G_1\G_1'^{-1}\G_2^{-1}\G_2'^{-1}, \G_3^{-1}\G_3'\G_3] &=& e. \label{12j}
\end{eqnarray}\end{tiny}

>From (\ref{12a}) and (\ref{12b}), we see that the group $\pi_1(\C\P^2 \setminus S)$ is generated by $\G_1,$
$\G_2,$ $\G_2'$ and $\G_3.$  Substituting into (\ref{12e}) and (\ref{12f}), we obtain
$$\G_1=\G_2'\G_2\G_1\G_2^{-1}\G_2'^{-1} \ \ \mbox{and} \ \ \G_3=\G_2'\G_2\G_3\G_2^{-1}\G_2^{-1},$$ or in other words
\begin{equation}\label{12k}
[\G_2'\G_2, \G_3]=[\G_2'\G_2, \G_1]=e.\end{equation} Hence, (\ref{12g}) reduces to \ $[\G_1, \G_3]=e,$ and
likewise (\ref{12h}), (\ref{12i}) and (\ref{12j}).

Applying (\ref{12k}), we get \ $\G_2'^{-1}\G_1\G_2'=\G_2\G_1\G_2^{-1},$ and hence
$\G_1\G_2'\G_1^{-1}=\G_1^{-1}\G_2\G_1,$ which implies that $\G_2'=\G_1^{-2}\G_2\G_1^2.$ Likewise,
$\G_2'=\G_3^{-2}\G_2\G_3^2.$

Hence, the image of $\G_2'$ in the quotient $G/\langle\G_1^2, \G_2^2, \G_3^2\rangle$ is equal to the image of
$\G_2.$ Thus $G/\langle\G_1^2, \G_2^2, \G_3^2\rangle \cong \left \langle\G_1, \G_2,
\G_3|\langle\G_1,\G_2\rangle=\langle\G_2, \G_3\rangle, [\G_1,\G_3], \G_1^2, \G_2^2, \G_3^2\right \rangle \cong
S_4.$  Hence, since the projection of the fundamental group of the branch curve complement to $S_4$ is onto, it
is an isomorphism; hence the fundamental group of $\Gal{X},$ which is the kernel of this map, is trivial.
\end{proof}

\subsubsection{The union of a cubic degeneration and a plane}\label{Cayley+} Consider a smooth quartic surface that degenerates
to a union of three planes meeting at a point, and one plane not passing through that point, as shown in Figure
\ref{ozen-haman+}.  For example, if the ideal of the degenerated surface is $(x,y)(x,z)(x,w)(w,y)$ or \begin{tiny} $(x^3w, x^3u,
x^2w^2, x^2wu, x^2zw, x^2zu, xzw^2, xzwu, x^2yw, x^2yu, xyw^2, xywu, xyzw, xyzu, yzw^2, yzwu),$\end{tiny} then one can
check by explicit computation that a generic deformation of this surface, such as
\begin{tiny} $(x^3w+t(x^4-y^4+z^4+w^4+u^4), x^3u+t(x^4+y^4-z^4+w^4+u^4), x^2w^2+t(x^4+y^4+z^4-w^4+u^4),
x^2wu+t(x^4+y^4+z^4+w^4-u^4), x^2zw+t(2x^4+y^4+z^4+w^4+u^4), x^2zu+t(x^4+2y^4+z^4+w^4+u^4),
xzw^2+t(x^4+y^4+2z^4+w^4+u^4), xzwu+t(x^4+y^4+z^4+2w^4+u^4), x^2yw+t(x^4+y^4+z^4+w^4+2u^4),
x^2yu-t(x^4-3*y^4+z^4+w^4+u^4), xyw^2+t(x^4+y^4-3z^4+w^4+u^4), xywu-t(x^4+y^4+z^4-3w^4+u^4),
xyzw-t(x^4+y^4+z^4+w^4-3u^4), xyzu-t(x^4+4y^4+z^4+w^4+u^4), yzw^2-t(x^4+y^4+4z^4+w^4+u^4),
yzwu-t(x^4+y^4+4z^4+w^4+u^4),$ \end{tiny} is smooth.

\begin{figure}[ht]
\epsfysize=2cm 
\begin{minipage}{\textwidth}
\begin{center}
\epsfbox{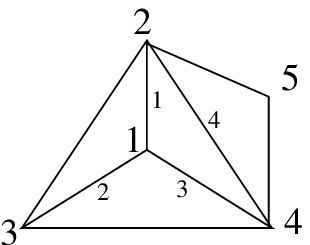}
\end{center}
\end{minipage}
\caption{Degeneration to three planes with a triple point and one other plane}\label{ozen-haman+}
\end{figure}

\begin{theorem}
The fundamental group of the Galois cover of this surface is $\mathbb{Z}^6 \semidirect \mathbb{Z}_2^2$.
\end{theorem}

\begin{proof}
 We break up the braid monodromy factorization into three components $\Delta_1,$ $\Delta_2$ and $\Delta_4.$
The factor $\Delta_1$ corresponding to the triple point $1$ in \ref{ozen-haman+} is computed in \cite{Cayley}
(where it is denoted $\widetilde{\Delta})$ as follows.

 \begin{eqnarray} \label{del-haman+1}
 \Delta_1 = & {(Z_{1' \; 3}^2)}^{Z^{-2}_{2', 3 \; 3'}} \cdot (Z_{1' \; 3'}^2)^{Z_{1' \; 3}^2 Z^{-2}_{2', 3 \;
3'}} \cdot {(Z_{1 \; 3}^2)}^{Z^{-2}_{2', 3 \; 3'}} \cdot (Z_{1 \; 3'}^2)^{Z_{1 \; 3}^2 Z^{-2}_{2', 3 \; 3'}}
\\ & \cdot {(Z_{2 \; 2'})}^{Z^{-2}_{1 \; 1', 2} \bar{Z}^2_{2, 3 \; 3'}} \cdot {(Z^3_{2, 3 \; 3'})}^ {Z^2_{2 \;
2'}} \cdot  Z^3_{1 \; 1', 2'} \cdot Z_{2 \; 2'} \nonumber
\end{eqnarray}

Note that the first, the fourth and the last two paths correspond to braids of branch points. The second and
third paths correspond to braids of cusps, and the rest correspond to braids of nodes.

\begin{figure}[ht]
\epsfxsize=6cm 
\begin{minipage}{\textwidth}
\begin{center}
\epsfbox{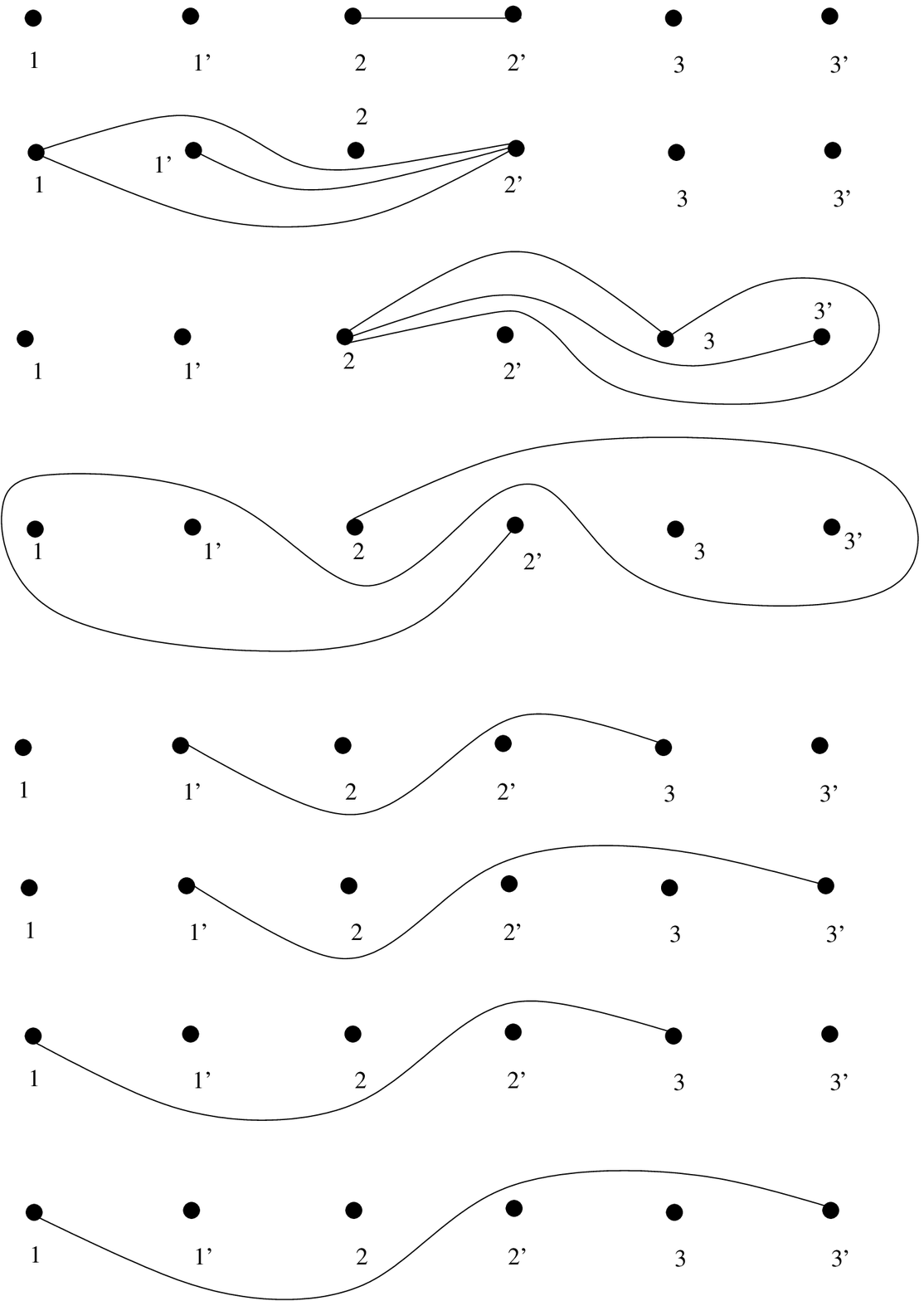}
\end{center}
\end{minipage}
\caption{$\Delta_1$ braids from (\ref{del-haman+1})}\label{Cayley+del1}
\end{figure}

$\Delta_1$ thus gives rise to the following relations on the generators of the fundamental group:
\begin{tiny}\begin{eqnarray}
{}\G_2 & = & \G_2', \label{Cayley+55} \\
{}\langle\G_1, \G_2'\rangle = \langle \G_1', \G_2'\rangle = \langle\G_1' \G_1 \G_1'^{-1}, \G_2'\rangle &=& e ,\label{Cayley+56}\\
{}\langle\G_2' \G_2 \G_2'^{-1}, \G_3\rangle = \langle\G_2' \G_2 \G_2'^{-1}, \G_3'\rangle =
\langle\G_2'\G_2\G_2'^{-1}, \G_3'\G_3\G_3'^{-1} \rangle & = & e, \label{Cayley+57}\\
{}\G_1^{-1} \G_1'^{-1} \G_2'^{-1} \G_3' \G_3 \G_2' \G_2 \G_2'^{-1} \G_3^{-1} \G_3'^{-1}
\G_2' \G_1' \G_1 & = & \G_2', \\
{}[\G_1', \G_2'^{-1} \G_3 \G_2'] = [\G_1', \G_2'^{-1} \G_3^{-1} \G_3' \G_3 \G_2'] & = & e, \label{Cayley+59}\\
{}[\G_1, \G_2'^{-1} \G_3 \G_2'] = [\G_1,\G_2'^{-1} \G_3^{-1} \G_3' \G_3 \G_2'] & = & e.
\end{eqnarray}\end{tiny}

Vertices $2$ and $4$ are equivalent, and are both equivalent to Vertex 5 of the (1,2) degeneration of
$\mathbb{CP}^1 \times \mathbb{CP}^1,$ as depicted in Figure \ref{(12v5)}.

The factor $\Delta_2$ is thus as given in Figure \ref{Cayley+del2},
\begin{figure}[h]
\epsfxsize=5cm 
\begin{minipage}{\textwidth}
\begin{center}
\epsfbox{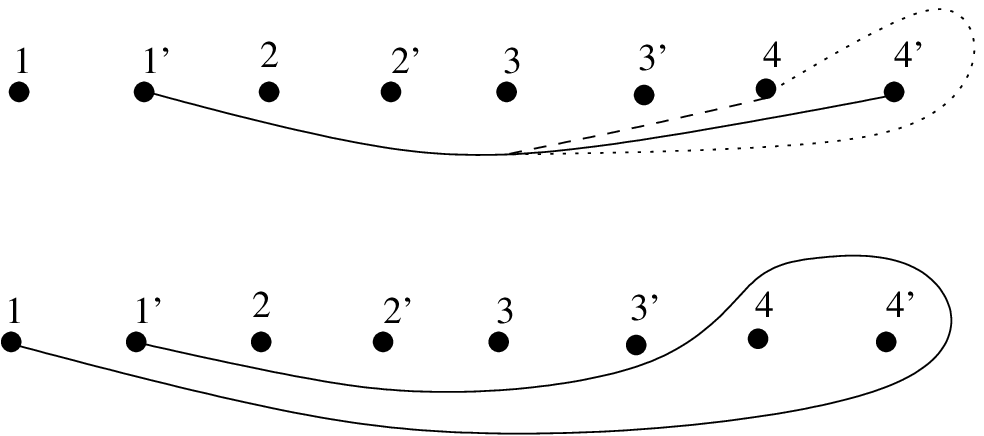}
\end{center}
\end{minipage}
\caption{$Z^3_{1', 44'} \cdot (Z_{11'}^{Z_{1',44'}^2})$}\label{Cayley+del2}
\end{figure}
and gives rise to the relations \begin{tiny}\begin{eqnarray}
{}&&\langle\G_1', \G_4\rangle = \langle \G_1', \G_4'\rangle = \langle\G_1', \G_4^{-1}\G_4'\G_4\rangle = e, \label{Cayley+61}\\
{}&&\G_1=\G_4'\G_4\G_1'\G_4^{-1}\G_4'^{-1}, \label{Cayley+62}
\end{eqnarray}\end{tiny}
while $\Delta_4$ is given in Figure \ref{Cayley+del3},
\begin{figure}[h]
\epsfxsize=5cm 
\begin{minipage}{\textwidth}
\begin{center}
\epsfbox{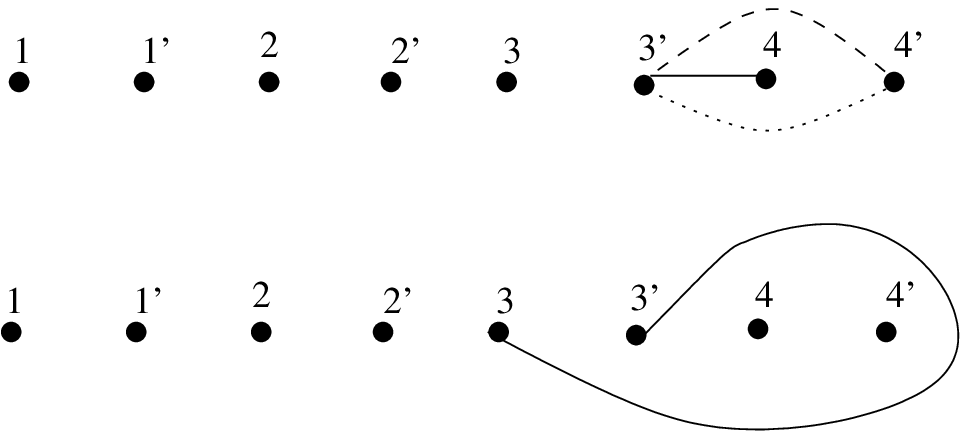}
\end{center}
\end{minipage}
\caption{$Z^3_{3', 44'}\cdot (Z_{33'}^{Z_{3',44'}}),$}\label{Cayley+del3}
\end{figure}
and gives rise to the relations \begin{tiny}\begin{eqnarray}
{}&&\langle\G_3', \G_4\rangle = \langle\G_3', \G_4'\rangle = \langle\G_3', \G_4^{-1}\G_4'\G_4\rangle = e, \label{Cayley+63}\\
{}&&\G_3=\G_4'\G_4\G_3'\G_4^{-1}\G_4'^{-1}. \label{Cayley+64}
\end{eqnarray}\end{tiny}

We also have the parasitic and projective relations.

The parasitic braids are explained in \cite[p. 616]{15}, and are
shown in Figure \ref{Cayley+parasitic}.

\begin{figure}[h]
\epsfxsize=5cm 
\begin{minipage}{\textwidth}
\begin{center}
\epsfbox{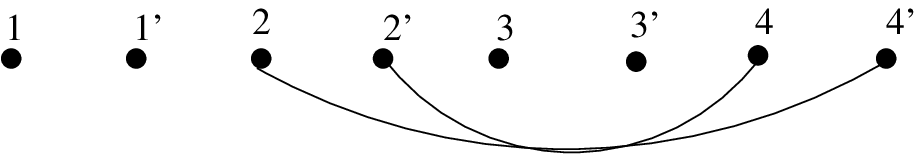}
\end{center}
\end{minipage}
\caption{$Z^2_{24},$ $Z^2_{2'4},$ $Z^2_{24'}$ and $Z^2_{2'4'}$}\label{Cayley+parasitic}
\end{figure}

The relations are \begin{tiny}\begin{eqnarray}
{}&&[\G_2, \G_4] = e, \label{Cayley+65}\\
{}&&[\G_2',\G_4] = e, \label{Cayley+66}\\
{}&&[\G_2, \G_4'] = e, \label{Cayley+67}\\
{}&&[\G_2',\G_4'] = e. \label{Cayley+68}
\end{eqnarray}\end{tiny}

The projective relation is \begin{tiny}\begin{equation*}\G_4'\G_4\G_3'\G_3\G_2'\G_2\G_1'\G_1
=e.\end{equation*}\end{tiny}

By (\ref{Cayley+62}), (\ref{Cayley+64}) and (\ref{Cayley+55}), it is clear that the group is generated by the
five generators $\G_1',$ $\G_2',$ $\G_3',$ $\G_4',$ and $\G_4,$ where the defining relations are
\begin{tiny}
\begin{eqnarray*}
{}\langle \G_1',\G_2'\rangle &=& e,\ \ \mbox{(\ref{Cayley+56})}\\
{}\langle \G_2',\G_3'\rangle &=& e,\ \ \mbox{from (\ref{Cayley+57}) and (\ref{Cayley+55})}\\
{}[\G_2',\G_4] &=& e,\ \ \mbox{(\ref{Cayley+66})}\\
{}[\G_2',\G_4'] &=& e,\ \ \mbox{(\ref{Cayley+68})}\\
{}\langle \G_1',\G_4'\rangle &=& e,\ \ \mbox{(\ref{Cayley+61})}\\
{}\langle \G_1',\G_4\rangle &=& e,\ \ \mbox{(\ref{Cayley+61})}\\
{}\langle \G_3',\G_4' \rangle &=& e,\ \ \mbox{(\ref{Cayley+63})}\\
{}\langle\G_3',\G_4 \rangle &=& e,\ \ \mbox{(\ref{Cayley+63})}\\
{}[\G_1', \G_2'^{-1}\G_3'\G_2'] &=& e. \ \ \mbox{(\ref{Cayley+59})}
\end{eqnarray*}
\end{tiny}

We map the generators \begin{equation*} \G_1' \mapsto (1,3), \  \G_2'\mapsto (1,2), \ \G_4'\mapsto (3,4), \ \G_4
\mapsto (3,4), \ \G_3' \mapsto (2,3),
\end{equation*} according to the configuration in Figure \ref{ozen-haman+}.

The diagrams (according to the convention of \cite{RTV}, given in \ref{V4}) for the map to $S_4$ are thus as
follows, and we have the additional relation $[\G_a', \G_2'^{-1}\G_3'\G_2']=e$, see Figure \ref{2-graphs}.

\begin{figure}[h]
\epsfxsize=9cm 
\begin{minipage}{\textwidth}
\begin{center}
\epsfbox{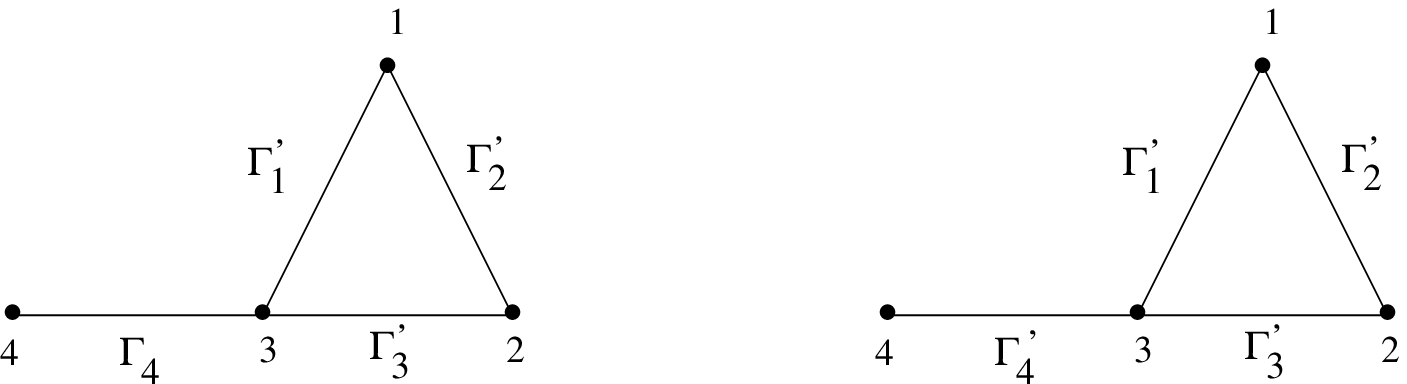}
\end{center}
\end{minipage}
\caption{}\label{2-graphs}
\end{figure}

Hence, the by \cite{RTV}[p.3], the kernel of the map to $S_4$ is generated by $\G_1'\G_2'^{-1}\G_3'\G_2, \, \G_2'\G_3'^{-1}\G_1'\G_3', \\ \G_3'\G_1'^{-1}\G_2'\G_1,$ arising from the cycle $\G_1', \, \G_2', \, \G_3',$  and by the conjugates of $(\G_1'\G_3'\G_4'\G_3')^2$ and $(\G_1'\G_3\G_4\G_3')^2,$ arising from the relation involving $3$ edges meeting at a vertex (relation (4) in \cite{RTV}[p.4]).

Since the cycle only involves the vertices $1,$ $2,$ and $3,$ and since $S_3$ is generated by two transpositions, it is enough to take the preimages of two transpositions that arise from the cyclic relation.  In particular, all elements of $K$ that arise from the cyclic relation are generated by $\G_1'\G_2'^{-1}\G_3'\G_2$ and $\G_2'\G_3'^{-1}\G_1'\G_3'$ and their conjugates.  Since in $G/\langle \G_i^2\rangle$ we have $\G_1'^2=(\G_2'^{-1}\G_3'\G_2)^2=(\G_2'\G_3'\G_2')^2=e,$ it follows from $[\G_a', \G_2'^{-1}\G_3'\G_2']=e$ that $(\G_1'\G_2'\G_3'\G_2')^2=e.$  Hence $\G_1'\G_2'\G_3'\G_2',$ and its conjugate $\G_2'\G_3'\G_1'\G_3',$ have order $2$ in $G/\langle \G_i^2\rangle.$  Hence, since these two elements commute, the subgroup of $K$ generated by these two elements is $\mathbb{Z}_2 \oplus \mathbb{Z}_2.$

The generators of $K$ that arise from three edges meeting in a vertex, and hence involve four vertices, are $(\G_1'\G_3'\G_4'\G_3')^2,$ $(\G_1'\G_3'\G_4\G_3')^2$ and their conjugates.  Since $S_4$ is generated by three transpositions, we need to take three conjugates of each of these elements. Hence, the group $K$ is isomorphic to $\mathbb{Z}^3 \oplus \mathbb{Z}^3 \semidirect \mathbb{Z}_2^2,$ or $\mathbb{Z}^6\semidirect \mathbb{Z}_2^2.$

\forget
Define the following set of new generators:
\begin{equation*}
s_1 := \G_1', \ \ s_2 := \G_2', \ \ s_4 := \G_4', \ \ s_4' := \G_4, \ \ s_3 :=  \G_2'^{-1}\G_3'\G_2'.
\end{equation*}

Then the map to $S_4$ is as follows:
\begin{equation*}
s_1 \mapsto (1,3), \ \ s_4 \mapsto (3,4), \ \ s_3\mapsto (1,3), \ \ s_4' \mapsto (3,4), \ \ s_2 \mapsto (1,2),
\end{equation*}
and the relations in $G/\langle \G_i^2\rangle$ are as follows: \begin{tiny}\begin{eqnarray}
{} s_i^2 &=& e,\ \ \mbox{from $\G_i^2=e$,}\label{Cayley+1}\\
{}\langle s_1,s_2\rangle &=& e,\ \ \mbox{from (\ref{Cayley+56}),} \label{Cayley+2}\\
{}\langle s_2,s_2^{-1}s_3s_2\rangle = \langle s_2, s_3 \rangle &=& e,\ \ \mbox{(from \ref{Cayley+57}) and (\ref{Cayley+55}),} \label{Cayley+3}\\
{}[s_2,s_4] &=& e,\ \ \mbox{from (\ref{Cayley+68}),} \label{Cayley+4}\\
{}[s_2,s_4'] &=& e,\ \ \mbox{from (\ref{Cayley+66}),}\\
{}\langle s_2^{-1}s_3s_2,s_4 \rangle &=& e,\ \ \mbox{from (\ref{Cayley+63}), which is equivalent to}\\
{}\langle s_3, s_2s_4s_2^{-1} \rangle &=& e,\ \ \mbox{so by (\ref{Cayley+4}) we obtain}\\
{}\langle s_3, s_4 \rangle &=& e. \ \ \mbox{Likewise,}\label{Cayley+5}\\
{}\langle s_3,s_4' \rangle &=& e,\ \ \mbox{from (\ref{Cayley+63}), and finally}\label{Cayley+6}\\
{}[s_1, s_3] &=& e, \ \ \mbox{from (\ref{Cayley+59}).}\label{Cayley+7}
\end{eqnarray}\end{tiny}
Hence, we have the following presentation for $G / <\G_i^2>:$ $$\left \langle s_1, s_2, s_3, s_4, s_4' | s_i^2,
\langle s_1, s_2 \rangle, \langle s_2, s_3 \rangle, \langle s_3, s_4 \rangle, \langle s_4, s_1 \rangle, \langle
s_3, s_4' \rangle, \langle s_4', s_1 \rangle, [s_1, s_3], [s_2, s_4], [s_2, s_4']\right \rangle.$$

Hence the Coxeter diagram for $G / <\G_i^2>$ is as given in Figure \ref{Cayley+Cox}.

\begin{figure}[ht]
\epsfxsize=3.5cm 
\begin{minipage}{\textwidth}
\begin{center}
\epsfbox{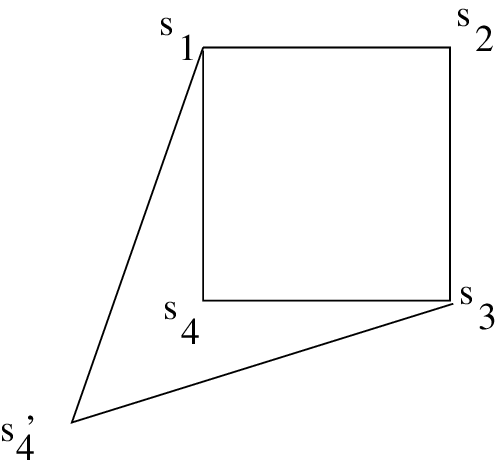}
\end{center}
\end{minipage}
\caption{Coxeter diagram of $G/<\G_i^2>$}\label{Cayley+Cox}
\end{figure}

\begin{lemma}\label{funny}
The kernel of the map from $G/\langle \G_i^2\rangle$ to $S_4$ is normally generated by $s_1s_3,$
$(s_1s_2s_3s_4s_3s_2)^2$ and $(s_1s_2s_3s_4's_3s_2)^2.$
\end{lemma}
\noindent {\bf Proof of Lemma \ref{funny}.} In $(\G_1'\G_3'\G_4'\G_3')^2,$  we substitute the $s_i$'s by
(\ref{Cayley+4}) to get $(s_1s_2s_3s_2s_4s_2s_3s_2)^2 = (s_1s_2s_3s_4s_3s_2)^2=e,$ and in
$(\G_1'\G_3'\G_4\G_3')^2,$ we get $(s_1s_2s_3s_2s_4's_2s_3s_2)^2=(s_1s_2s_3s_4's_3s_2)^2=e $. $\Box$ To get,
from the normal set of generators from Lemma \ref{funny}, a generating set for the kernel of the map $\rho :
G/<\G_i^2> \rightarrow S_4,$ we need to add the following two conjugates of $(s_1s_2s_3s_4s_3s_2)^2:$
$(s_2s_3s_4s_1s_4s_3)^2$ and $(s_3s_4s_1s_2s_1s_2)^2,$ the following two conjugates of
$(s_1s_2s_3s_4's_3s_2)^2:$ $(s_2s_3s_4's_1s_4's_3)^2$ and $(s_3s_4's_1s_2s_1s_4')^2,$ and the following
conjugate of the involution $s_1s_3:$ $s_2s_1s_3s_2.$

Hence, it follows that the kernel is $\mathbb{Z}^6 \semidirect \mathbb{Z}_2^2,$ where the copy of $\mathbb{Z}^6$
is generated by $(s_1s_2s_3s_4s_3s_2)^2,$ $(s_2s_3s_4s_1s_4s_3)^2,$ $(s_3s_4s_1s_2s_1s_4)^2,$
$(s_1s_2s_3s_4's_3s_2)^2,$ $(s_2s_3s_4's_1s_4's_3)^2,$ and $(s_3s_4's_1s_2s_1s_4')^2,$ and the generators of
$\mathbb{Z}_2^2$ are $s_1s_3$ and $s_2s_1s_3s_2.$
\forgotten
\end{proof}

\subsubsection{The $4$-point}\label{4-pt}  The last possible degeneration of a quartic surface is to a plane
arrangement with a $4$-point, as in Figure \ref{4-ptsurf}.

\begin{figure}[ht]
\epsfxsize=3cm 
\begin{minipage}{\textwidth}
\begin{center}
\epsfbox{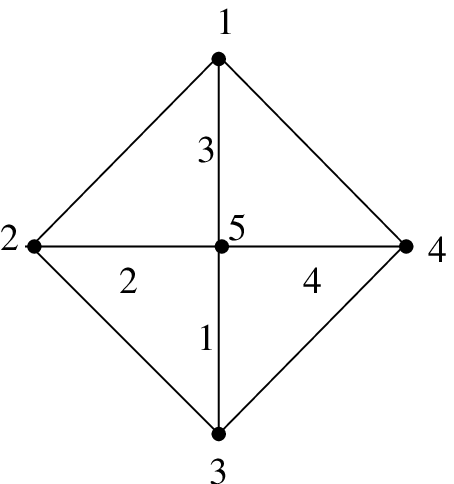}
\end{center}
\end{minipage}
\caption{Surface degeneration diagram}\label{4-ptsurf}
\end{figure}

Again, this surface can be smooth or singular.  For example, in $\C\P^4,$ consider the surfaces defined by the
ideal $(xz+tu(x+y+z+w), yw+tu(x+y+z-w)).$  For generic $t$ the surface is smooth, but when $t=0$ we get the four
planes defined by the ideal $(xz, yw).$

\begin{theorem}
The fundamental group of the Galois cover of the $4$-point surface is $\mathbb{Z}_2^3$.
\end{theorem}

\begin{proof}
The degenerated surface has one $4$-point and four $2$-points (nodes). Then each node regenerates to two branch points.
 This gives rise to the braid monodromy factors $\varphi_1  =  Z_{1 \; 1'} \cdot Z_{1 \; 1'}, \ \ \varphi_2  =  Z_{2 \; 2'} \cdot Z_{2 \; 2'}, \ \ \varphi_3  =  Z_{3 \; 3'} \cdot Z_{3 \; 3'},  \ \ \varphi_4 =  Z_{4 \; 4'} \cdot Z_{4 \; 4'}$, which are depicted in Figure \ref{4-pt-braids1}.

\begin{figure}[ht]
\epsfxsize=6cm 
\begin{minipage}{\textwidth}
\begin{center}
\epsfbox{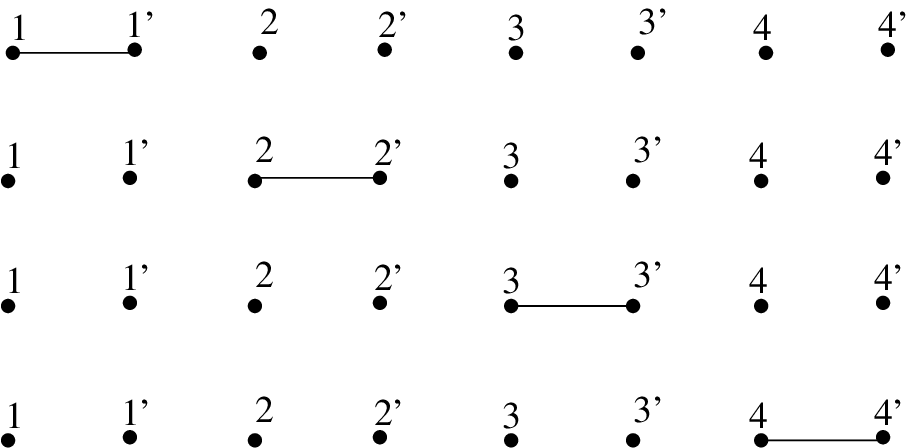}
\end{center}
\end{minipage}
\caption{$2$-point braids}\label{4-pt-braids1}
\end{figure}

The $4$-point gives rise to the following braid monodromy factors, as computed in \cite{Ogata}. These braids appear in order in Figure \ref{4-pt-braids2}.
\begin{eqnarray}
{}\varphi_5 & = & (Z^3_{1', 2 \; 2'} \cdot {Z^3_{3 \; 3',4}} \cdot h_1 \cdot h_2 \cdot {(Z^2_{1' \; 4})}^{Z_{1', 2 \; 2'}^2} \cdot Z^2_{1 \; 4}) \cdot \\ & & (Z^3_{1, 2 \; 2'} \cdot {(Z^3_{3 \; 3',4'})}^{Z^{-2}_{4 \; 4'}} \cdot h_3 \cdot h_4 \cdot {(Z^2_{1 \; 4'})}^{Z_{1, 2 \; 2'}^2 Z^{-2}_{4 \; 4'}}
\cdot {(Z^2_{1' \; 4'})}^{Z_{1 \; 1'}^{-2} Z^{-2}_{4 \; 4'}} ).\nonumber
\end{eqnarray}

\begin{figure}[ht]
\epsfxsize=6cm 
\begin{minipage}{\textwidth}
\begin{center}
\epsfbox{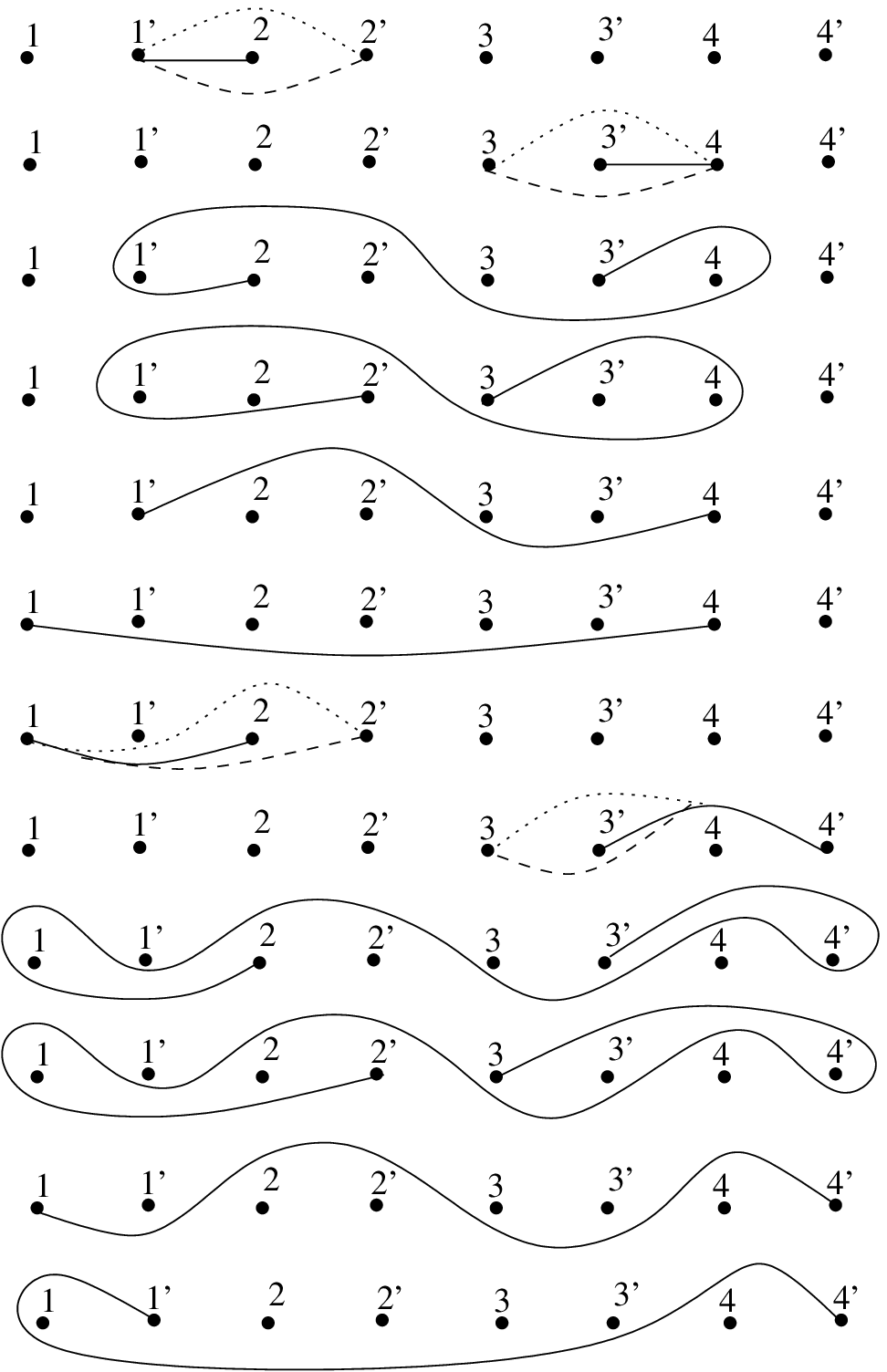}
\end{center}
\end{minipage}
\caption{$4$-point braids}\label{4-pt-braids2}
\end{figure}

These braids give rise to the relations
\begin{tiny}\begin{eqnarray}
{}&&\G_1 =  \G_1'\label{4pt69}\\
{}&&\G_2 = \G_2' \label{4pt70}\\
{}&&\G_3 = \G_3' \label{4pt71}\\
{}&&\G_4 = \G_4'  \label{4pt72}\\
{}&&\langle\G_1', \G_2 \rangle = \langle \G_1', \G_2'\rangle = \langle\G_1' \G_2^{-1}\G_2'\G_2\rangle = e \label{4pt73}\\
{}&&\langle\G_3, \G_4\rangle = \langle \G_3', \G_4\rangle = \langle\G_3' \G_3 \G_3'^{-1}, \G_4\rangle = e \label{4pt74}\\
{}&&[\G_2'\G_2\G_1'\G_2^{-1}\G_2'^{-1},\G_4] = e\label{4pt75}\\
{}&&[\G_1,\G_4] = e\label{4pt76}\\
{}&&\langle\G_1, \G_2\rangle = \langle \G_1, \G_2'\rangle = \langle\G_1, \G_2^{-1}\G_2'\G_2\rangle = e \label{4pt77}\\
{}&&\langle\G_3, \G_4^{-1}\G_4'\G_4\rangle = \langle \G_3', \G_4^{-1}\G_4'\G_4\rangle =
\langle\G_3' \G_3 \G_3'^{-1}, \G_4^{-1}\G_4'\G_4\rangle = e \label{4pt78}\\
{}&&[\G_2'\G_2\G_1\G_2^{-1}\G_2'^{-1},\G_4^{-1}\G_4'\G_4] = e\label{4pt79}\\
{}&&[\G_1^{-1}\G_1'\G_1,\G_4^{-1}\G_4'\G_4] = e\label{4pt80}\\
{}&&\G_2'\G_2\G_1'\G_2\G_1'^{-1}\G_2^{-1}\G_2'^{-1} = \G_4\G_3'\G_4^{-1}\label{4pt81}\\
{}&&\G_2'\G_2\G_1'\G_2'\G_1'^{-1}\G_2^{-1}\G_2'^{-1} = \G_4\G_3'\G_3\G_3'^{-1}\G_4^{-1}\label{4pt82}\\
{}&&\G_2'\G_2\G_1\G_2\G_1^{-1}\G_2^{-1}\G_2'^{-1} = \G_4^{-1}\G_4'\G_4\G_3'\G_4^{-1}\G_4'^{-1}\G_4\label{4pt83}\\
{}&&\G_2'\G_2\G_1\G_2'\G_1^{-1}\G_2^{-1}\G_2'^{-1} =
\G_4^{-1}\G_4'\G_4\G_3'\G_3\G_3'^{-1}\G_4^{-1}\G_4'^{-1}\G_4. \label{4pt84}
\end{eqnarray}\end{tiny}
We also have the projective relation
\begin{tiny}\begin{equation}\label{4ptprojrel}
\G_4'\G_4\G_3'\G_3\G_2'\G_2\G_1'\G_1=e.
\end{equation}\end{tiny}

Relations (\ref{4pt69})-(\ref{4pt72}) simplify the relations as follows:
\begin{eqnarray*}
{}&&\langle\G_1, \G_2 \rangle = e \\
{}&&\langle\G_3, \G_4\rangle =  e \\
{}&&[\G_2^2\G_1\G_2^{-2},\G_4] = e \\
{}&&[\G_1,\G_4] = e \\
{}&&\G_2^2\G_1\G_2\G_1^{-1}\G_2^{-2} = \G_4\G_3\G_4^{-1}\\
{}&&\G_4^2\G_3^2\G_2^2\G_1^2=e.
\end{eqnarray*}

In the group $G/<\G_i^2>$, we get
\begin{eqnarray*}
{}&&\langle\G_1, \G_2 \rangle = e \\
{}&&\langle\G_3, \G_4\rangle =  e \\
{}&&[\G_1,\G_4] = e \\
{}&&\G_1\G_2\G_1^{-1} = \G_4\G_3\G_4^{-1}.
\end{eqnarray*}
Using these relations, we have:
$$e=\langle\G_3, \G_4\rangle = \langle\G_4 \G_3 \G_4^{-1}, \G_4\rangle = \langle\G_1 \G_2 \G_1^{-1}, \G_4\rangle =
\langle\G_1 \G_2 \G_1^{-1}, \G_1 \G_4 \G_1^{-1} \rangle = \langle\G_2, \G_4\rangle.$$
Since $\G_3 = \G_4^{-1}\G_1\G_2\G_1^{-1}\G_4$, we get that $G/<\G_i^2>$ is generated by $\G_1, \G_2, \G_4$, with the relations
$\G_i^2 =e, \langle\G_1, \G_2 \rangle = \langle\G_2, \G_4 \rangle = [\G_1,\G_4] = e$. Therefore,  $G/<\G_i^2> \cong S_4$.

Since $S_4$ is a finite group, the only map onto $S_4$ is the identity map, and hence the kernel of the map is trivial.
\end{proof}

\section{Appendix: Calculations of Galois covers of $\mathbb{CP}^1 \times \mathbb{CP}^1$}\label{App}
\subsection{The case ($a=2$, $b=2$)} \label{F0-22}

We now consider the degeneration of $\mathbb{CP}^1 \times \mathbb{CP}^1$ of bidegree $(a=2; \, b=2),$ as shown
in Figure \ref{(2,2)}.  We expect $\deg \Delta_{16}^2=240$ conditions on the $16$ Van Kampen generators $\G_1,
\G_1', \cdots, \G_8, \G_8'.$
\begin{figure}[h]
\epsfysize=4cm 
\begin{minipage}{\textwidth}
\begin{center}
\epsfbox{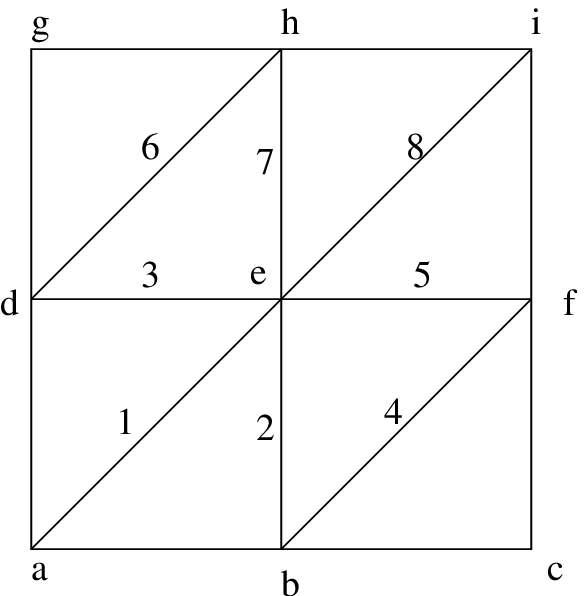}
\end{center}
\end{minipage}
\caption{$\C\P^1 \times \C\P^1$ $(2,2)$-degeneration}\label{(2,2)}
\end{figure}

We construct the braid relations from the degeneration.  Vertex $a$ (resp., vertex $c$) is equivalent to vertex
$1$ (resp., vertex $3$) of the $(1,2)$-degeneration from \ref{F0-12}, so giving rise to the single conditions
\begin{tiny}\begin{eqnarray}
{}\G_1=\G_1'\label{1'}\\
{}\G_8=\G_8'.\label{8'}
\end{eqnarray}\end{tiny}

Vertex $b$ and vertex $d$ are equivalent to vertex $2$ of the $(1,2)$-degeneration, giving rise to the sets of
degree $10$ conditions \begin{tiny}
\begin{eqnarray}
{}\G_4' &=& \G_4\G_2'\G_2\G_4\G_2^{-1}\G_2'^{-1}\G_4^{-1}\label{4b}\\
{}\langle a,\G_4 \rangle &=& e, \ \ \mbox{where $a=\G_2,$ $\G_2'$ or $\G_2^{-1}\G_2'\G_2$}\label{<2*,4>}\\
{}\G_6' &=& \G_4\G_2'\G_2\G_6\G_2^{-1}\G_2'^{-1}\G_4^{-1}\label{6d}\\
{}[a,\G_6] &=& e,\ \ \mbox{where $a=\G_2,$ $\G_2'$ or $\G_2^{-1}\G_2'\G_2.$}\label{<2*,6>}
\end{eqnarray}
\end{tiny}

The braids for the central vertex $e$ include those shown in Figures \ref{eb15} and \ref{eb610},
 \begin{figure}[ht!]
\epsfxsize=9cm 
\begin{minipage}{\textwidth}
\begin{center}
\epsfbox{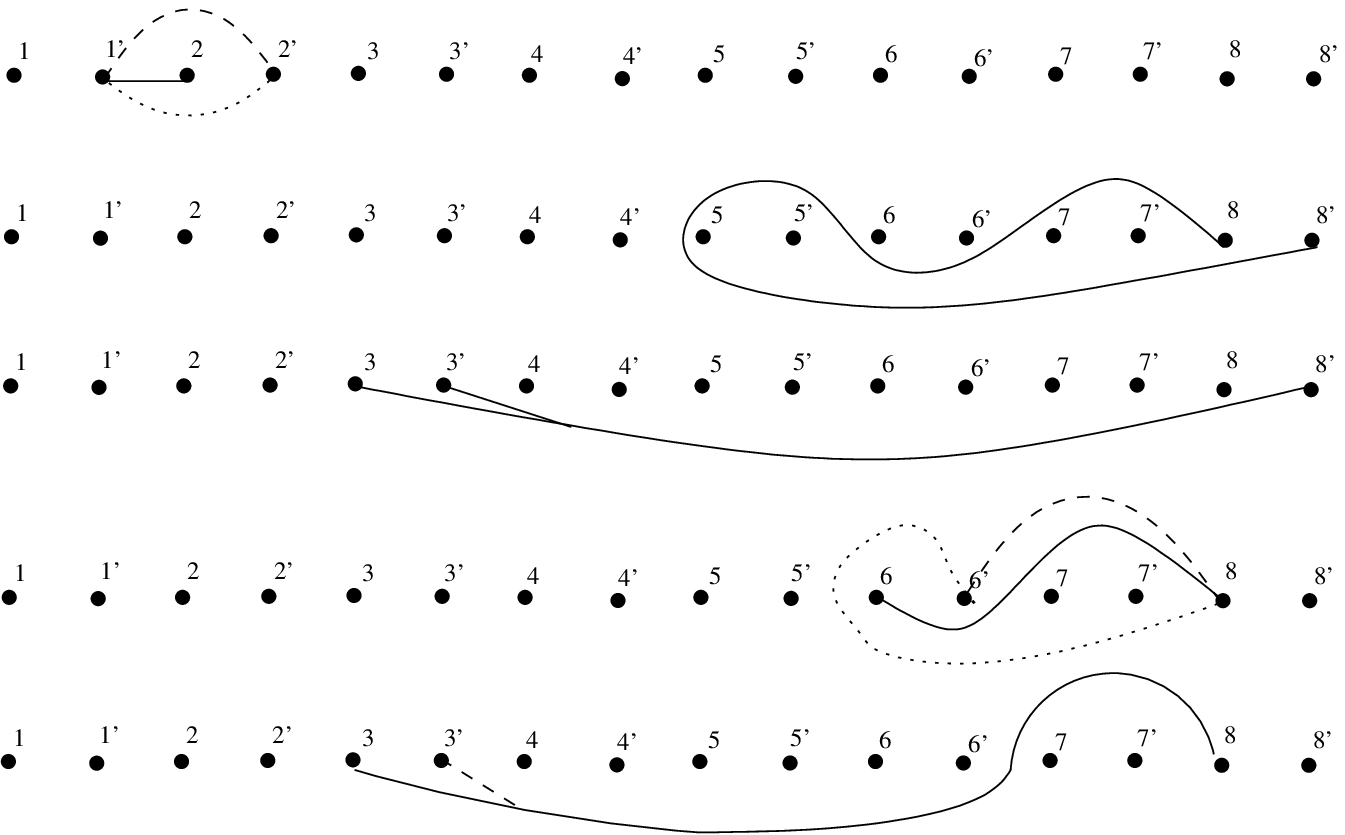}
\end{center}
\end{minipage}
\caption{$Z_{1',22}^3, (Z_{88'})^{Z_{55',8}^{-2}Z_{77',8}^{-2}}, Z_{33',8}^2, \bar{Z}_{55',8}^3,
(Z_{33',8}^2)^{Z_{77',8}^{-2}}$}\label{eb15}
\end{figure}

\begin{figure}[h]
\epsfxsize=9cm 
\begin{minipage}{\textwidth}
\begin{center}
\epsfbox{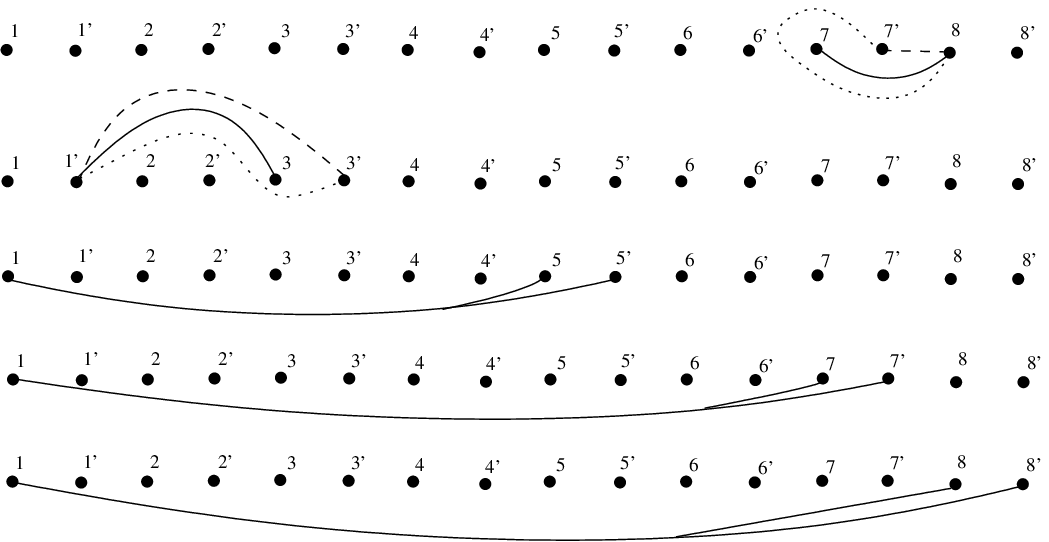}
\end{center}
\end{minipage}
\caption{$Z_{77',8}^3, \bar{Z}_{1',33'}^3, Z_{1,55'}^2, Z_{1,77'}^2, Z_{1,88'}^2$}\label{eb610}
\end{figure}
\noindent as well as those shown in Figures \ref{ef13} and
\ref{ef46}, which are all conjugated by $Z_{77',8}^{-2}Z_{1',
22'}^2$, those shown in Figures \ref{ef79} and \ref{ef1012}, which
are all conjugated by $Z_{77'}^{-1}Z_{22}^{-1},$ and those
shown in Figures \ref{eg13} and \ref{eg46}, which are all
conjugated by $Z_{1', 22'}^2.$ $Z_{77',8}^{-2}Z_{1', 22'}^2,$
\begin{figure}[tbp]
\epsfxsize=9cm 
\begin{minipage}{\textwidth}
\begin{center}
\epsfbox{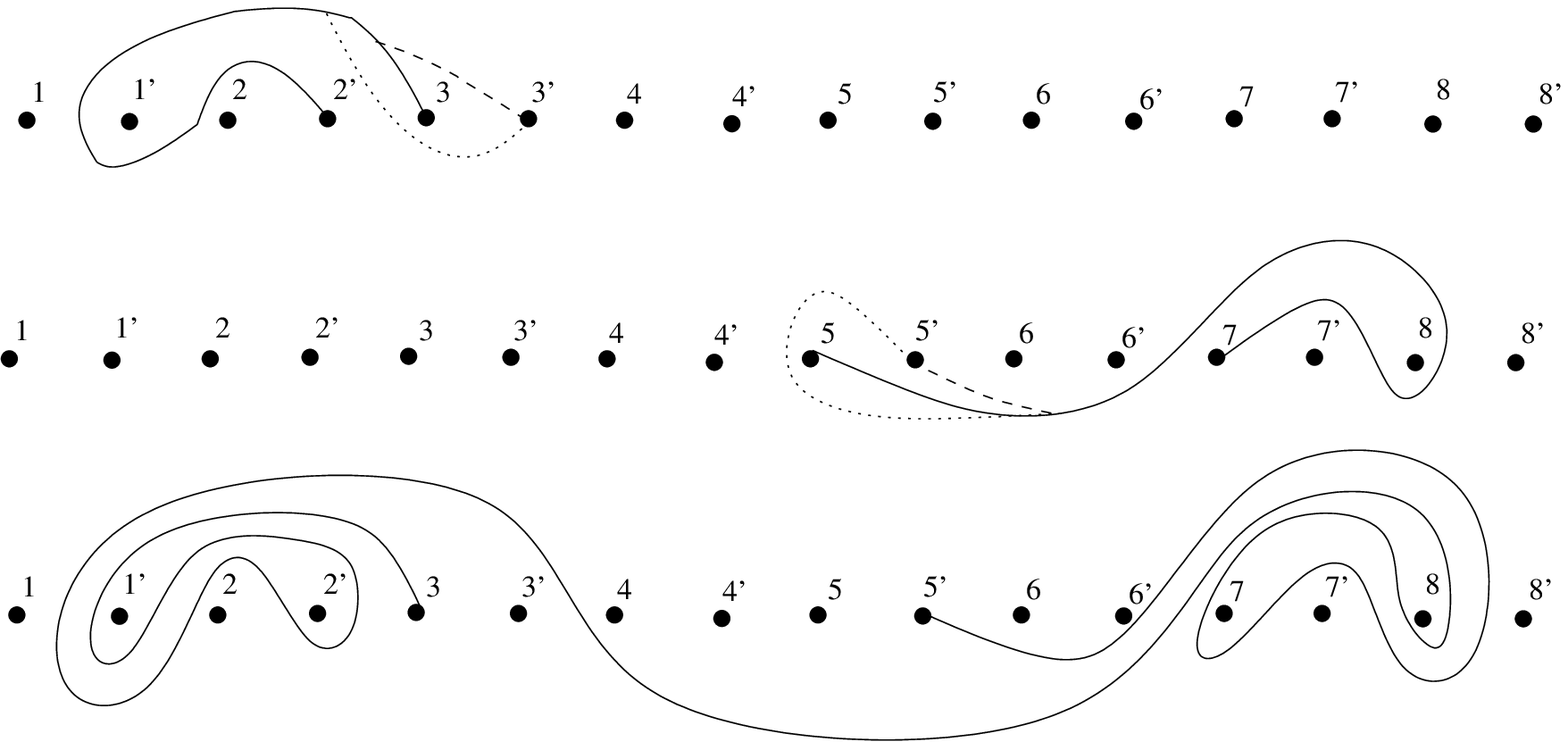}
\end{center}
\end{minipage}
\caption{$(Z_{2',33'}^3)^{Z_{77',8}^{-2}Z_{1', 22'}^2}, (Z_{55',7}^3)^{Z_{77',8}^{-2}Z_{1', 22'}^2},
(\wt{Z}_{3,5'})^{Z_{77',8}^{-2}Z_{1', 22'}^2}$}\label{ef13}
\end{figure}
\begin{figure}[tbp]
\epsfxsize=9cm 
\begin{minipage}{\textwidth}
\begin{center}
\epsfbox{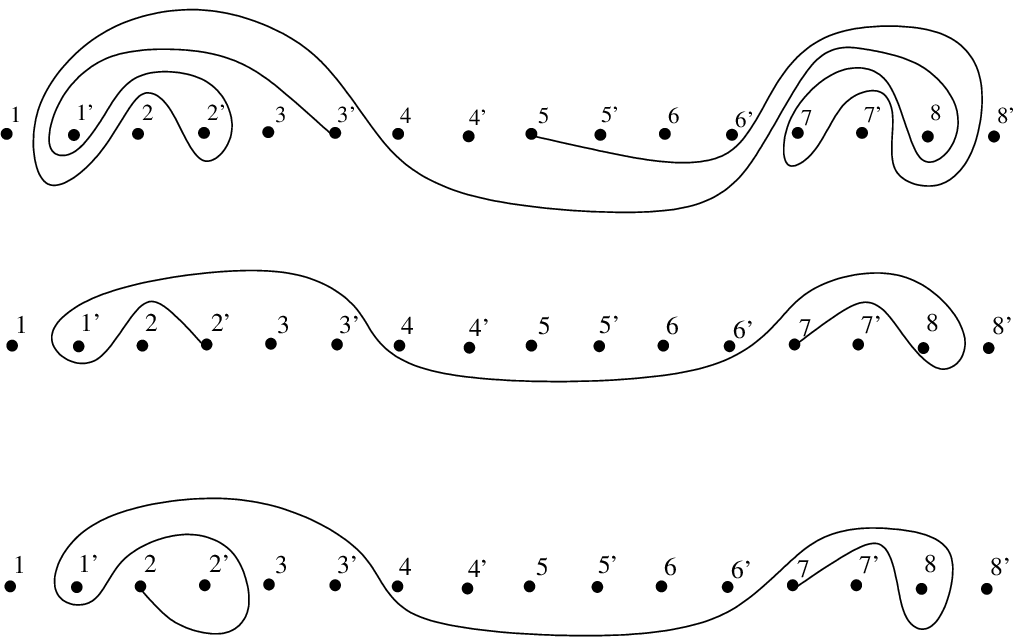}
\end{center}
\end{minipage}
\caption{\tiny{$(\wt{Z}_{3',5})^{Z_{77',8}^{-2}Z_{1', 22'}^2}, (Z_{2'7}^2)^{Z_{2,33'}^2Z_{77',8}^{-2}Z_{1',
22'}^2}, (Z_{27}^2)^{Z_{77',8}^{-2}Z_{1', 22'}^2}$}}\label{ef46}
\end{figure}
\begin{figure}[tbp]
\epsfxsize=9cm 
\begin{minipage}{\textwidth}
\begin{center}
\epsfbox{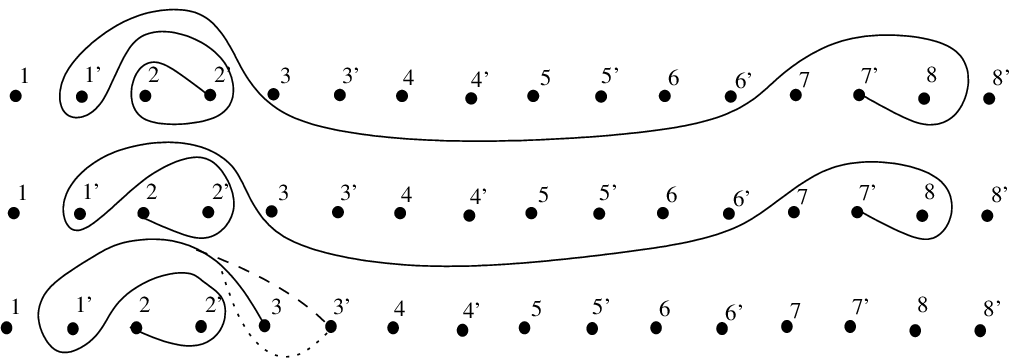}
\end{center}
\end{minipage}
\caption{{\tiny{${(Z_{2',7}^2)}^{Z_{77'}^{-1}Z_{22'}^{-1}Z_{2',33}^2Z_{77',8}^{-2}Z_{1', 22'}^2},
{(Z_{2,27'}^2)}^{Z_{77'}^{-1}Z_{22'}^{-1}Z_{77',8}^{-2}Z_{1', 22'}^2},
{(Z_{2',33'}^3)}^{Z_{77'}^{-1}Z_{22'}^{-1}Z_{77',8}^{-2}Z_{1',22'}^2}$}}}\label{ef79}
\end{figure}
\begin{figure}[b]
\epsfxsize=9cm 
\begin{minipage}{\textwidth}
\begin{center}
\epsfbox{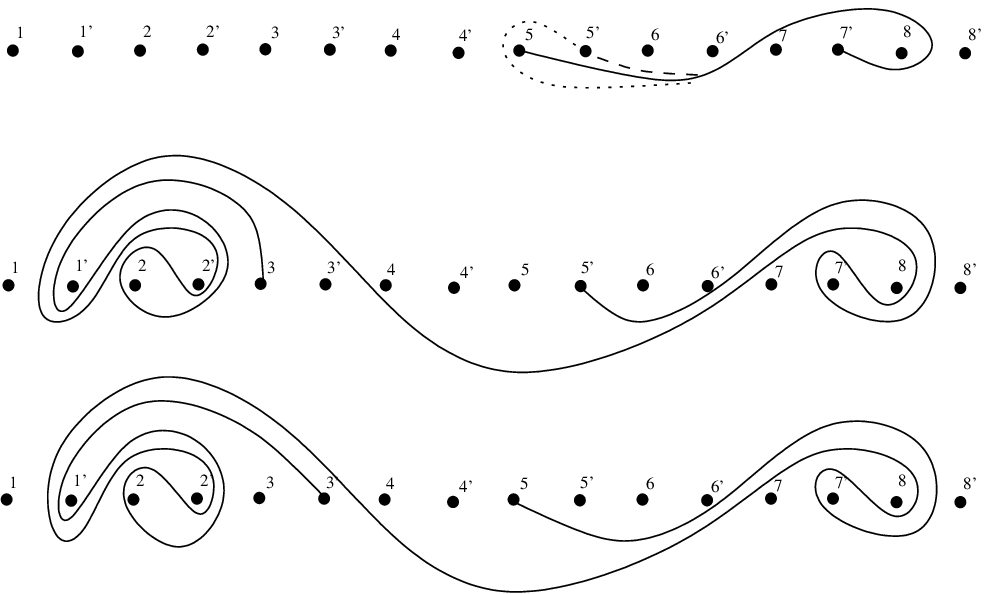}
\end{center}
\end{minipage}
\caption{\tiny{$(Z_{55',7}^3)^{Z_{77'}^{-1}Z_{22}^{-1}Z_{77',8}^{-2}Z_{1',22'}^2},
(\wt{Z}_{3,5'})^{Z_{77'}^{-1}Z_{22}^{-1}Z_{77',8}^{-2}Z_{1', 22'}^2},
(\wt{Z}_{3',5})^{Z_{77'}^{-1}Z_{22}^{-1}Z_{77',8}^{-2}Z_{1', 22'}^2}$}}\label{ef1012}
\end{figure}
\begin{figure}[t]
\epsfxsize=9cm 
\begin{minipage}{\textwidth}
\begin{center}
\epsfbox{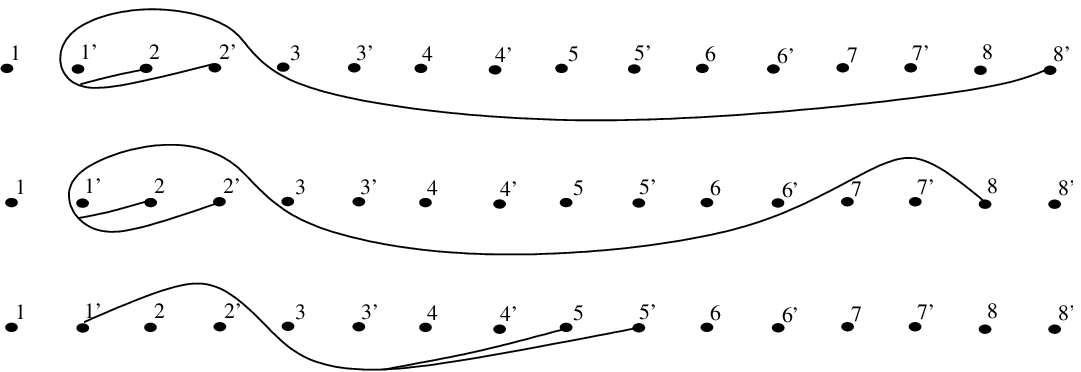}
\end{center}
\end{minipage}
\caption{$(Z_{22',8}^2)^{Z_{1', 22'}^2}, (Z_{22',8}^2)^{Z_{77',8}^{-2}Z_{1', 22'}^2}, (Z_{1,55'}^2)^{Z_{1',
22'}^2}$}\label{eg13}
\end{figure}
\begin{figure}[h]
\epsfxsize=9cm 
\begin{minipage}{\textwidth}
\begin{center}
\epsfbox{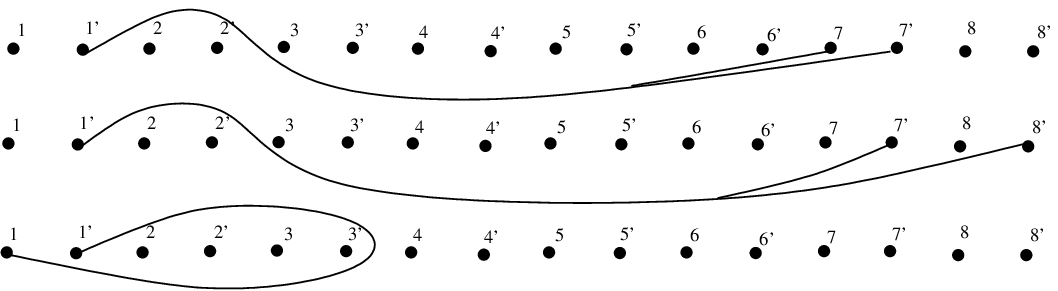}
\end{center}
\end{minipage}
\caption{$(Z_{1',77'}^2)^{Z_{1', 22'}^2}, (Z_{1',88'}^2)^{Z_{1', 22'}^2}, (Z_{1,1'})^{Z_{1', 33'}^2Z_{1',
22'}^2}$}\label{eg46}
\end{figure}

 The vertex thus gives rise to the following conditions:
\begin{tiny}
\begin{eqnarray}
{}\langle\G_1',a\rangle = e, \ \ \mbox{where $a=\G_2,$ $\G_2'$ or $\G_2^{-1}\G_2'\G_2,$}\label{<1',2*>} \\
{}\G_8' = \G_5^{-1}\G_5'^{-1}\G_7^{-1}\G_7'^{-1}\G_8\G_7'\G_7\G_5'\G_5, \label{8'**}\\
{}[b,\G_8'] = e, \ \ \mbox{where $b=\G_3$ or $\G_3',$}\label{[3*,8']}\\
{}\langle \G_6^{-1}\G_6'^{-1}\G_7^{-1}\G_7'^{-1}\G_8\G_7'\G_7\G_6'\G_6, c\rangle = e,\ \ \mbox{where $c=\G_5,$ $\G_5',$ or $\G_5^{-1}\G_5'\G_5,$}\label{<87'76'6,5*>}\\
{}[\G_7^{-1}\G_7'^{-1}\G_8\G_7'\G_7, b] = e,\ \ \mbox{ where $b=\G_3$ or $\G_3',$}\label{[87'7,3*]}\\
{}\langle d, \G_8 \rangle = e, \ \ \mbox{where $d=\G_7,$ $\G_7',$ or $\G_7^{-1}\G_7'\G_7,$}\label{<7*,8>}\\
{}\langle \G_2'\G_2\G_1'\G_2^{-1}\G_2'^{-1}, b \rangle = e, \ \ \mbox{where $b=\G_3,$ $\G_3'$ or $\G_3^{-1}\G_3'\G_3,$}\label{<1'-(2'2), 3*>}\\
{}[\G_1, c] = e,\ \ \mbox{where $c=\G_5$ or $\G_5',$}\label{[1,5*]}\\
{}[\G_1, d] = e,\ \ \mbox{where $d=\G_7$ or $\G_7',$}\label{[1,7*]}\\
{}[\G_1, e] = e,\ \ \mbox{where $e=\G_8$ or $\G_8',$}\label{[1,8*]}
\end{eqnarray}

\begin{eqnarray}
{}\langle \G_2'\G_2\G_1'\G_2^{-1}\G_2'\G_2\G_1'^{-1}\G_2^{-1}\G_2'^{-1}, b \rangle = e, \ \ \mbox{where $b=\G_3,$ $\G_3'$ or $\G_3^{-1}\G_3'\G_3,$}\label{<(2'-(1'-2))2',3*>}\\
{}\langle c, \G_7^{-1}\G_7'^{-1}\G_8^{-1}\G_7'\G_7\G_7'^{-1}\G_8\G_7'\G_7 \rangle = e, \ \ \mbox{where $c=\G_5,$ $\G_5'$ or $\G_5^{-1}\G_5'\G_5,$}\label{<5*, (7(87'))7>}\\
{}\G_3'\G_3\G_2'\G_2\G_1'\G_2^{-1}\G_2'\G_2\G_1'^{-1}\G_2^{-1}\G_2'^{-1}\G_3\G_2'\G_2\G_1'\G_2^{-1}\G_2'^{-1}\G_2\G_1'^{-1}\G_2^{-1}\G_2'^{-1}\G_3^{-1}\G_3'^{-1} \label{3,5'*}\\
{}\nonumber = \G_7^{-1}\G_7'^{-1}\G_8^{-1}\G_7'\G_7\G_7'^{-1}\G_8\G_7'\G_7\G_5'\G_7^{-1}\G_7'^{-1}\G_7'\G_7^{-1}\G_7'^{-1}\G_8\G_7'\G_7,\\
{}\G_3'\G_3\G_2'\G_2\G_1'\G_2^{-1}\G_2'\G_2\G_1'^{-1}\G_2^{-1}\G_2'^{-1}\G_3^{-1}\G_3'\G_3\G_2'\G_2\G_1'\G_2^{-1}\G_2'^{-1}\G_2\G_1'^{-1}\G_2^{-1}\G_2'^{-1}\G_3^{-1}\G_3'^{-1}\\
{}\nonumber = \G_7^{-1}\G_7'^{-1}\G_8^{-1}\G_7'\G_7\G_7'^{-1}\G_8\G_7'\G_7\G_5\G_7^{-1}\G_7'^{-1}\G_8^{-1}\G_7'\G_7^{-1}\G_7'^{-1}\G_8\G_7'\G_7, \label{3',5*}\\
{}[\G_3'\G_3\G_2'\G_2\G_1'\G_2^{-1}\G_2'\G_2\G_1'^{-1}\G_2^{-1}\G_2'^{-1}\G_3^{-1}\G_3'^{-1},\G_7^{-1}\G_7'^{-1}\G_8^{-1}\G_7'\G_7\G_7'^{-1}\G_8\G_7'\G_7] = e, \label{[1'-(3'32'1'-2),7((87')7)]}\\
{}[\G_2'\G_2\G_1'\G_2^{-1}\G_2'^{-1}\G_2\G_2'\G_2\G_1'^{-1}\G_2^{-1}\G_2'^{-1},\G_7^{-1}\G_7'^{-1}\G_8^{-1}\G_7'\G_7\G_7'^{-1}\G_8\G_7'\G_7]
= e, \\ \label{[2-(1'-(22')),7(87'7)]}
{}[\G_2'\G_2\G_1'\G_2^{-1}\G_2'^{-1}\G_2^{-1}\G_2'\G_2\G_2'\G_2\G_1'^{-1}\G_2^{-1}\G_2'^{-1},
\G_7^{-1}\G_7'\G_8^{-1}\G_7'\G_8\G_7'\G_7]= e, \label{[(2'2)-(1'-(2'2)), 7'(87'7)]}\\
{}[\G_2'\G_2\G_1'\G_2^{-1}\G_2'^{-1}\G_2\G_2'\G_2\G_1'^{-1}\G_2^{-1}\G_2'^{-1}, \G_7^{-1}\G_7'^{-1}\G_8^{-1}\G_7'\G_8\G_7'\G_7] = e,\label{[2-(1'-(2'2)),7'(87'7)]}\\
{}\langle \G_2'\G_2\G_1'\G_2^{-1}\G_2'^{-1}\G_2\G_2'\G_2\G_1'^{-1}\G_2^{-1}\G_2'^{-1}, b\rangle = e, \ \ \mbox{where $b=\G_3,$ $\G_3'$ or $\G_3^{-1}\G_3'\G_3,$}\label{<2-(1'(-2'2)),3*>}\\
{}\langle c, \G_7^{-1}\G_7'^{-1}\G_8^{-1}\G_7'\G_8\G_7'\G_7 \rangle = e, \ \
\mbox{where $c=\G_5,$ $\G_5'$ or $\G_5^{-1}\G_5'\G_5,$}\label{<5*,7'(87'7)>}\\
{}\G_3'\G_3\G_2'\G_2\G_1'\G_2^{-1}\G_2'^{-1}\G_2\G_2'\G_2\G_1'^{-1}\G_2^{-1}
\G_2'^{-1}\G_3\G_2'\G_2\G_1'\G_2^{-1}\G_2'^{-1}\G_2^{-1}\G_2'\G_2\G_1'^{-1}\G_2^{-1}
\G_2'^{-1}\G_3^{-1}\G_3'^{-1} \label{3,5'**}\\
{}=\G_7^{-1}\G_7'^{-1}\G_8^{-1}\G_7'\G_8\G_7'\G_7\G_5'\G_7^{-1}\G_7'^{-1}
\G_8^{-1}\G_7'^{-1}\G_8\G_7'\G_7, \nonumber \\
{}\G_3'\G_3\G_2'\G_2\G_1'\G_2^{-1}\G_2'^{-1}\G_2\G_2'\G_2\G_1'^{-1}\G_2^{-1}\G_2'^{-1}\G_3^{-1}\G_3'\G_3
\G_2^{-1}\G_2'^{-1}\G_2^{-1}\G_2'\G_2\G_1'^{-1}\G_2^{-1}\G_2'^{-1}\G_3^{-1}\G_3'^{-1} \label{3',5**}\\
{}=\G_7^{-1}\G_7'^{-1}\G_8^{-1}\G_7'\G_8\G_7'\G_7\G_5\G_7^{-1}\G_7'^{-1}\G_8^{-1}\G_7'\G_8\G_7'\G_7,\nonumber \\
{}[\G_2'\G_2\G_1'a\G_1'^{-1}\G_2'^{-1}, \G_8'] = e, \label{[2*-(2'21'),8']}\\
{}[\G_2'\G_2\G_1'a\G_1'^{-1}\G_2^{-1}\G_2'^{-1}, \G_7'\G_7'\G_8\G_7'\G_7] = e, \
 \ \mbox{where $a=\G_2$ or $\G_2',$}
\label{[2*-(2'21'), 8(7'7)]}\\
{}[\G_2'\G_2\G_1'\G_2^{-1}\G_1'^{-1}, c]= e, \ \ \mbox{where $c=\G_5$ or $\G_5',$} \label{[1'-(2'2),5*]}\\
{}[\G_2'\G_2\G_1'\G_2^{-1}\G_2'^{-1}, d] = e,\ \ \mbox{where $d=\G_7$ or $\G_7',$}\label{[1'-(2'2), 7*]}\\
{}[\G_2'\G_2\G_1'\G_2^{-1}\G_2'^{-1}, f] = e,\ \ \mbox{where $f=\G_8$ or $\G_8',$}\label{[1'-(2'2), 8*]}\\
{}\G_3'\G_3\G_2'\G_2\G_1'\G_2^{-1}\G_2'^{-1}\G_3^{-1}\G_3'^{-1} = \G_1. \label{1=1'-(3'32'2)}
\end{eqnarray}
\end{tiny}

Vertex $f$ and vertex $h$ are equivalent to vertex $5$ of the $(1,2)$ degeneration, giving rise to the sets of
degree $10$ conditions \begin{tiny}
\begin{eqnarray}
\G_4 &=& \G_5'\G_5\G_4'\G_5^{-1}\G_5'^{-1},\label{4d}\\
\langle \G_4',c \rangle &=& e,\ \ \mbox{where $c=\G_5,$ $\G_5'$ or $\G_5^{-1}\G_5'\G_5$},\label{<4',5*>}\\
\G_6 &=& \G_7'\G_7\G_6'\G_7^{-1}\G_7'^{-1},\label{7d}\\
\langle \G_6', d \rangle &=& 1,\ \ \mbox{where $d=\G_7,$ $\G_7'$ or $\G_7^{-1}\G_7'\G_7.$}\label{<6,7*>}
\end{eqnarray}\end{tiny}

Finally, for the parasitic intersections, we have the following braids:

\begin{figure}[ht!]
\epsfxsize=9cm 
\begin{minipage}{\textwidth}
\begin{center}
\epsfbox{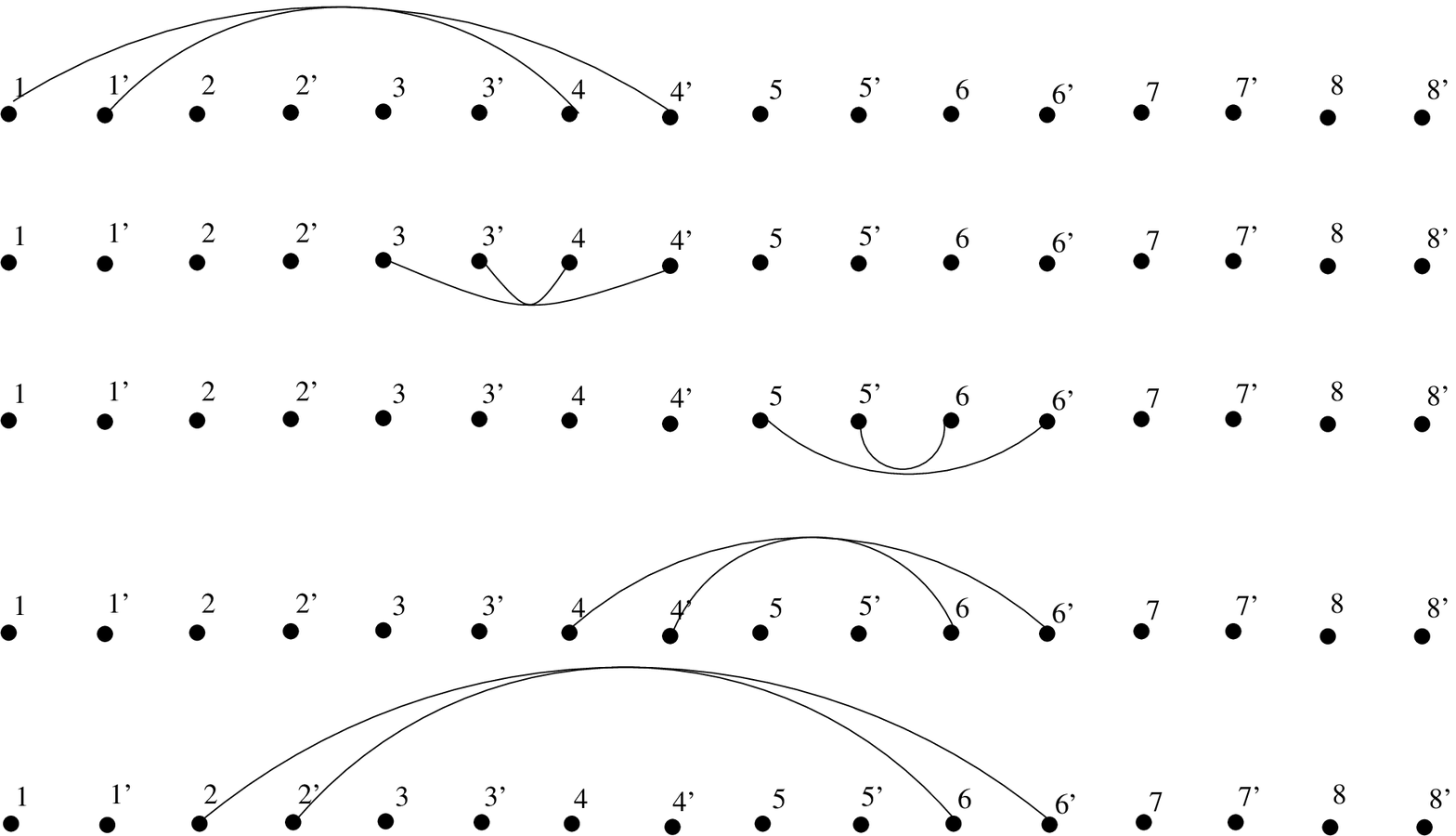}
\end{center}
\end{minipage}
\caption{\tiny{$\bar{Z}_{11',44'}^2, \bar{Z}_{33',44'}^2, Z_{55',66'}^2, Z_{44',66'}^2,
Z_{22',66'}^2$}}\label{para15}
\end{figure}
\begin{figure}[ht!]
\epsfxsize=9cm 
\begin{minipage}{\textwidth}
\begin{center}
\epsfbox{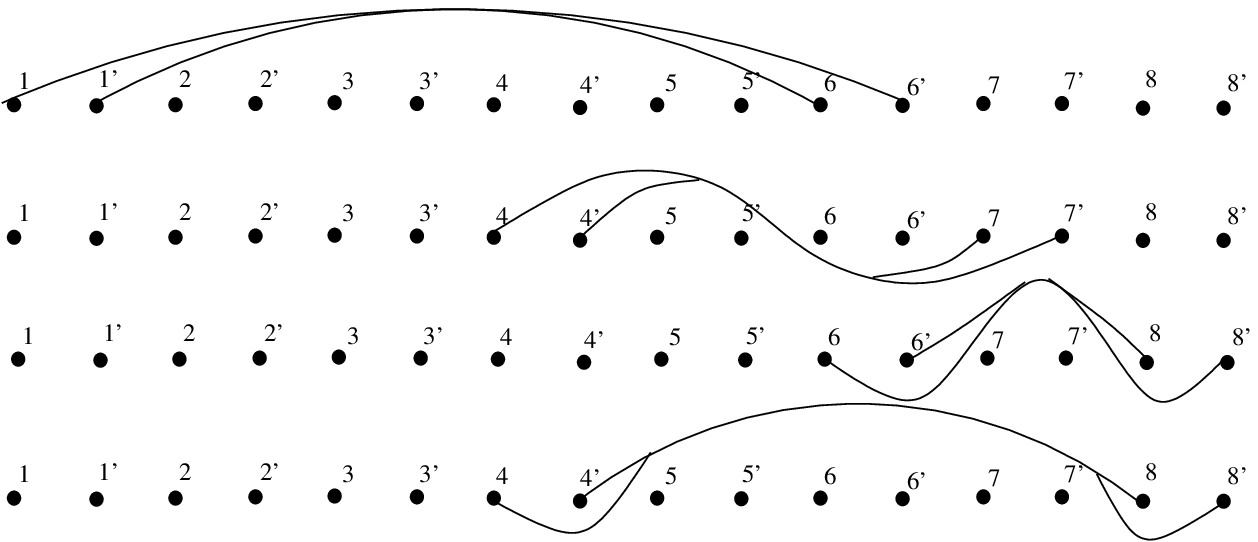}
\end{center}
\end{minipage}
\caption{\tiny{$\bar{Z}_{11',66'}^2, (Z_{44',77'}^2)^{\bar{Z}^{2}_{4',55'}},
(Z^2_{66',88'})^{\bar{Z}^{2}_{77',8}}, (Z^2_{44',88'})^{\bar{Z}^2_{4',8}}$}}\label{para69}
\end{figure}

Thus the relations are \begin{tiny}
\begin{eqnarray}
{}[\G_3'\G_3\G_2'\G_2g\G_2^{-1}\G_2'^{-1}\G_3^{-3}\G_3'^{-1}, h] = e,\\
{}\nonumber \mbox{where $g=\G_1'\G_1\G_1'^{-1}$ or $\G_1',$ and $h=\G_4$ or $\G_4',$}\label{[1*-(3'32'2),4*]}\\
{}[b, h] = e,\; \mbox{where $b=\G_3$ or $\G_3'$ and $h=\G_4$ or $\G_4',$}\label{[3*,4*]}\\
{}[d, i] = e, \; \mbox{where $d=\G_5$ or $\G_5'$ and $i=\G_6$ or $\G_6',$}\label{[5*,6*]}\\
{}[\G_5'\G_5g\G_5^{-1}\G_5'^{-1},i] = e, \; \mbox{where $g=\G_4'\G_4\G_4'^{-1}$ or $\G_4'$ and $i=\G_6$ or $\G_6^{-1}\G_6'\G_6,$}\label{[4*-(5'5),6*]}\\
{}[\G_5'\G_5\G_4'\G_4\G_3'\G_3a\G_3^{-1}\G_3'^{-1}\G_4^{-1}\G_4'^{-1}\G_5^{-1}\G_5'^{-1}, i] = e,\\
{}\nonumber \mbox{where $a=\G_2'\G_2\G_2'^{-1}$ or $\G_2'$ and $i=\G_6$ or
$\G_6^{-1}\G_6'\G_6.$}\label{[2*-(5'54'43'3), 6*]}\\
{}[\G_5'\G_5\G_4'\G_4\G_3'\G_3\G_2'\G_2g\G_2^{-1}\G_2'^{-1}\G_3^{-1}\G_3'^{-1}\G_4^{-1}\G_4'^{-1}\G_5^{-1}\G_5'^{-1}, i] = e, \\
{}\nonumber \mbox{where $g=\G_1'$, $\G_1'\G_1\G_1'^{-1}$; $i=\G_6$, $\G_6^{-1}\G_6'\G_6,$}\label{[1*-(5'54'43'32'2),6*]}\\
{}[\G_5'\G_5h\G_5^{-1}\G_5'^{-1}, d]=e,\; \mbox{where $h=\G_4$ or $\G_4'$; $d=\G_7$ or $\G_7',$}\label{[4*-(5'5),7*]}\\
{}[\G_7'\G_7i\G_7^{_1}\G_7'^{-1}, f]=e,\; \mbox{where $i=\G_6$ or $\G_6'$; $f=\G_8$ or $\G_8',$}\label{[6*-(7'7),8*]}\\
{}[\G_7'\G_7\G_6'\G_6\G_5'\G_5h\G_5^{-1}\G_5'^{-1}\G_6^{-1}\G_6'^{-1}\G_7^{-1}\G_7'^{-1}, f]=e,\\
{}\nonumber \mbox{where $h=\G_4$ or $\G_4'$; $f=\G_8$ or
$\G_8'.$}\label{[4*-(7'76'65'5),8*]}\end{eqnarray}\end{tiny}

Finally, we also have the projective relation \begin{tiny}\begin{equation} \label{projrel}
\G_8'\G_8\G_7'\G_7\G_6'\G_6\G_5'\G_5\G_4'\G_4\G_3'\G_3\G_2'\G_2\G_1'\G_1=e.
\end{equation}\end{tiny}

\begin{theorem} The fundamental group of the Galois cover is normally generated by $[\G_2, \G_4 \G_5
\G_4].$\end{theorem}
\begin{proof}

The proof will use the following key lemma repeatedly.

\begin{lemma}\label{keylem}
Let $G$ be a group containing $x,$ $y$ and $y'$ such that $\langle x, y \rangle = \langle x, y' \rangle =
\langle x, y^{-1}y'y \rangle = e,$ \ \ $[x, y'y]=e$ \ \ $x^2=y^2=y'^2=e$.   Then $y=y'$.
\end{lemma}

\noindent {\bf Proof of Lemma \ref{keylem}.} Since $\langle x, y\rangle =e,$ $x^{-1}yx=yxy^{-1}.$  Since
$\langle x, y'\rangle =e,$ $x^{-1}y'x=y'xy'^{-1}.$ Since $[x,y'y]=e,$ it follows that
$(x^{-1}y'x)(x^{-1}yx)=x^{-1}y'yx=y'y.$  On the other hand, $(x^{-1}y'x)(x^{-1}yx)=y'xy'^{-1}yxy^{-1},$ so
$y'xy'^{-1}yxy^{-1}=y'y$.

Since $y^2=1,$ $y^{-1}=y.$  Hence $xy'^{-1}yx=e,$ since $x=x^{-1}.$

Since $y'^2=1,$ $y'^{-1}=y'.$  Hence $e=x^{-1}y'^{-1}yx=x^{-1}y'yx=y'y,$ so $y'=y^{-1}=y.$ \hspace{1cm} $\Box$

Using the lemma, we continue the proof that the fundamental group of the Galois cover, the kernel of the map
$G\rightarrow S_8$ is trivial modulo the subgroup normally generated by the $\G_i^2$, by showing that each
$\G_i$ is isomorphic to $\G_i'$ modulo the squares.

Starting from the fact that $\G_1'=\G_1$ and $\G_8'=\G_8,$ we can substitute the relation (\ref{4d}) into
(\ref{[4*-(5'5),6*]}) to obtain
$$[\G_4,\G_6]=[\G_4,\G_6']=e.$$

>From the relations (\ref{6d}) and (\ref{<2*,6>}), we get that $\G_6' = \G_4\G_6\G_4^{-1}.$  Since $\G_4$
commutes with $\G_6,$ we get $\G_6'=\G_6.$

>From the relation (\ref{7d}), using the fact that $\G_6'=\G_6,$ it follows that $[\G_6, \G_7'\G_7]=e,$ so by
(\ref{<6,7*>}), we get that $\G_7'=\G_7$ using Lemma \ref{keylem}.

By (\ref{8'**}), since $\G_8'=\G_8,$ and $\G_7'=\G_7,$ modulo $\G_8^2$ and $\G_7^2$ we get
$\G_8=\G_5^{-1}\G_5'^{-1}\G_8\G_5'\G_5$. \\ Therefore,  $[\G_8,\G_5'\G_5]=e.$

By (\ref{<87'76'6,5*>}), substituting $\G_6=\G_6'$ and $\G_7=\G_7',$ we conclude that
\begin{equation}\label{<8,5*>}
\langle \G_8,c\rangle=e,
\end{equation}
where $c=\G_5,$ $\G_5'$ or $\G_5^{-1}\G_5'\G_5.$  Thus, using Lemma \ref{keylem} again, we conclude that
$\G_5=\G_5'.$

By substituting $\G_5=\G_5'$ into (\ref{4d}) and annihilating the squares, it follows that $\G_4=\G_4'.$

Consider the relation (\ref{4b}), and substitute $\G_4=\G_4'$ to get $[\G_4, \G_2'\G_2]=e.$  Thus, by
(\ref{<2*,4>}) and Lemma \ref{keylem}, it follows that $\G_2=\G_2'.$

We now consider the projective relation (\ref{projrel}).  Since $\G_i=\G_i'$ for all $i \neq 3$ and $\G_i^2=e$
for all $i,$ we conclude that $\G_3=\G_3'.$  Hence $\G_i=\G_i'$ for all $i.$

The defining relations (\ref{4b})-(\ref{[4*-(7'76'65'5),8*]}) reduce to
\begin{eqnarray*}
[\G_1,\G_4]=[\G_1,\G_5] = [\G_1,\G_6] = [\G_1,\G_7]=[\G_1,\G_8] = [\G_2, \G_6] = [\G_2,\G_6]= [\G_2,\G_8]\\
=[\G_3,\G_4]=[\G_3,\G_8] = [\G_4,\G_6] = [\G_4,\G_7] =[\G_4,\G_8]=[\G_5,\G_6] = [\G_6,\G_8] = e,
\end{eqnarray*}

\begin{eqnarray*}
\langle \G_1,\G_2 \rangle = \langle \G_1,\G_3 \rangle = \langle \G_1\G_2\G_1^{-1}, \G_3 \rangle = \langle
\G_2,\G_4 \rangle =\langle \G_4,\G_5 \rangle \\ = \langle \G_5,\G_8 \rangle = \langle \G_6,\G_7 \rangle =
\langle\G_7,\G_8 \rangle =\langle \G_5, \G_7\G_8\G_7^{-1}\rangle = e.
\end{eqnarray*}

>From (\ref{3,5'*}) and (\ref{3',5*}), (\ref{3,5'**}) and (\ref{3',5**}), substituting $\G_i$ for $\G_i',$ we get
$$\G_1\G_2\G_1\G_3\G_1\G_2\G_1= \G_8\G_7\G_8\G_5\G_8\G_7\G_8.$$ Hence,
$\G_2\G_1\G_3\G_1\G_2=\G_7\G_8\G_5\G_8\G_7,$ since $[\G_1, \G_7\G_8\G_5\G_8\G_7]=[\G_8,
\G_1\G_2\G_1\G_3\G_1\G_2\G_1]=e$.

Since $[\G_i,\G_8]=1$ for $1\leq i \leq 3,$ we have
\begin{eqnarray*}
e=[\G_1\G_2\G_1\G_3\G_1\G_2\G_1,\G_8] &=& [\G_8\G_7\G_7\G_5\G_8\G_7\G_8, \G_8]\\
=[\G_7\G_8\G_5\G_8\G_7, \G_8] &=& [\G_8\G_5\G_8, \G_7\G_8\G_7]\\
= [\G_8\G_5\G_8,\G_8\G_7\G_8] &=& [\G_5,\G_7].
\end{eqnarray*}

In the same way,
\begin{eqnarray*}
e=[\G_8\G_7\G_8\G_8\G_8\G_7\G_8,\G_1] &=& [\G_1\G_2\G_1\G_3\G_1\G_2\G_1, \G_1]\\
=[\G_2\G_1\G_3\G_1\G_2, \G_1] &=& [\G_1\G_3\G_1, \G_2\G_1\G_2]\\
= [\G_1\G_3\G_1,\G_1\G_2\G_1] &=& [\G_3,\G_2].
\end{eqnarray*}

>From $\G_2\G_1\G_3\G_1\G_2=\G_7\G_8\G_5\G_8\G_7$,  we get
$$\G_3=\G_1\G_2\G_7\G_8\G_5\G_8\G_2\G_1 \ \ \ \mbox{and}\ \ \ \G_5=\G_8\G_7\G_2\G_1\G_3\G_1\G_2\G_7\G_8.$$

Thus, \begin{eqnarray*} \langle \G_3,\G_6\rangle = \langle \G_1\G_2\G_7\G_8\G_5\G_7\G_2\G_1, \G_6\rangle =
\langle \G_2,\G_6\rangle =e,\\
\langle \G_2,\G_5 \rangle = \langle \G_2,\G_8\G_7\G_2\G_1\G_3\G_1\G_2\G_7\G_8\rangle = \langle \G_2,
\G_1\G_3\G_1 \rangle = \langle \G_2, \G_3\G_1\G_3\rangle =e,\\
\langle \G_3,\G_7 \rangle = \langle\G_1\G_2\G_6\G_8\G_5\G_8\G_2\G_2\G_1, \G_7\rangle = \langle \G_8,\G_7 \rangle
=e.
\end{eqnarray*}

Hence we obtain the Coxeter group with Dynkin diagram given in
Figure \ref{RTV1}, which, by \cite[Theorem 2.3, p.4]{RTV}, has
dual diagram given in Figure \ref{RTV2}, and can be mapped to
$S_8,$ such that the kernel is normally generated by
$[\G_1,\G_4\G_5\G_4].$
\end{proof}
\begin{figure}[ht]
\epsfysize=4.5cm 
\begin{minipage}{\textwidth}
\begin{center}
\epsfbox{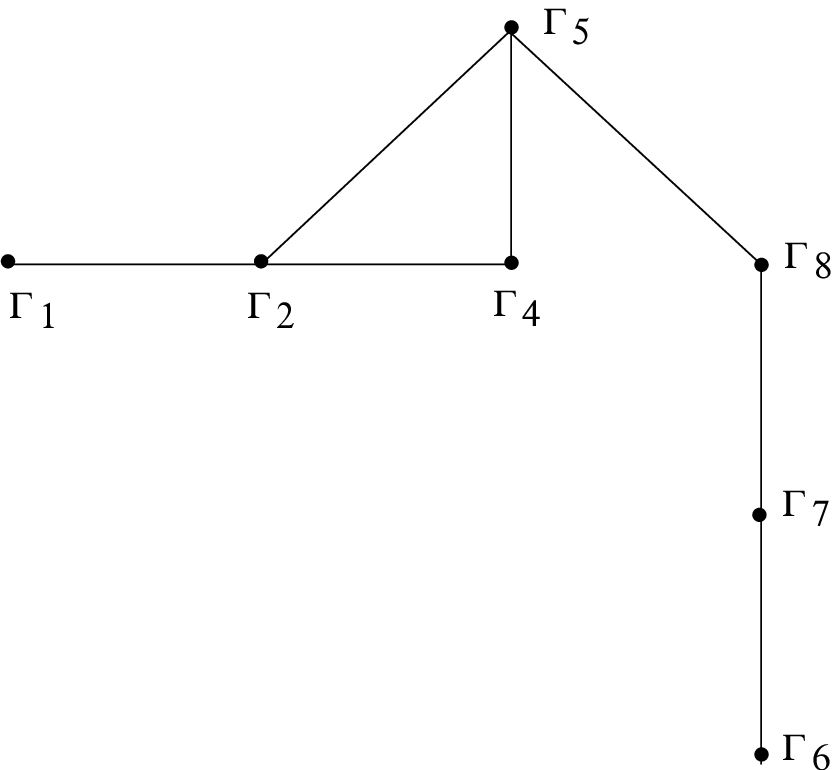}
\end{center}
\end{minipage}
\caption{The Coxeter diagram}\label{RTV1}
\end{figure}

\begin{figure}[ht]
\epsfysize=4cm 
\begin{minipage}{\textwidth}
\begin{center}
\epsfbox{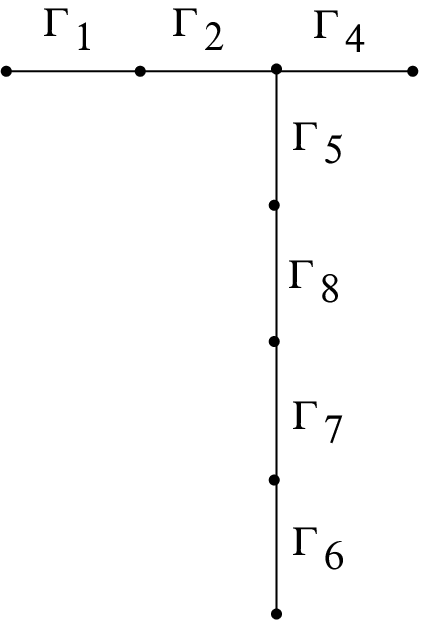}
\end{center}
\end{minipage}
\caption{The dual diagram}\label{RTV2}
\end{figure}

\begin{remark}
This result can also be deduced from \cite{lidtke}.
\end{remark}

\subsection{The case ($a=1$, $b=n$)} \label{F0-1n}

We now provide a more detailed proof of the following theorem (which appeared as an unproved corollary in
\cite{Ogata}):
\begin{theorem} The Galois cover of $(1,n)$-embedding of $\C\P^1 \times \C\P^1$ is simply connected,
for all $n.$
\end{theorem}

\begin{proof}
We first construct the braid relations.

All the vertices resemble those of the $(1, 2)$ degeneration. The two corners 1 and 2n+2 look like the corner
vertices 1 and 6 of the $(1,2)$ degeneration, and give rise to the following equations:
\begin{tiny}\begin{eqnarray}
{}\G_1&=&\G_1',\label{1}\\
\G_{2n-1} &=& \G_{2n-1}'.\label{2n-1}
\end{eqnarray}\end{tiny}
The lower vertices look like vertex 2 of the $(1,2)$ degeneration, and give rise to the following equations:
\begin{tiny}\begin{eqnarray}
{}\langle \G_{2i}, \G_{2i+1}\rangle = \langle \G_{2i}',\G_{2i-1}\rangle &=& \langle \G_{2i}^{-1}\G_{2i}'\G_{2i},\G_{2i+1}\rangle = e \label{ibraids}\\
{}\G_{2i-1}'&=&\G_{2i+1}\G_{2i}'\G_{2i}\G_{2i+1}\G_{2i}^{-1}\G_{2i}'^{-1}\G_{2i+1}^{-1},\label{i-ident}
\end{eqnarray}\end{tiny}
for $i=1, \cdots, n-1.$ The upper vertices appear like vertex 5 of the (1,2) degeneration, and give rise to the
following equations: \begin{tiny}\begin{eqnarray}
{}\langle \G_{2j-1}', \G_{2j}\rangle = \langle \G_{2j-1}', \G_{2j}'\rangle &=& \langle \G_{2j-1}',\G_{2j}^{-1}\G_{2j}'\G_{2j}\rangle = e \label{jbraids}\\
{}\G_{2j-1}'&=&\G_{2j}'\G_{2j}\G_{2j-1}'\G_{2j}^{-1}\G_{2j}'^{-1},\label{j-ident}
\end{eqnarray}\end{tiny}
for $j=1, \cdots, n-1.$ The ``parasitic'' intersections give rise to the following equations:
\begin{tiny}\begin{eqnarray} {}[a, b]&=&e,  \label{kl}
\end{eqnarray}\end{tiny}
where $a=\G_l$ or $\G_l'$ and $b=\G_{2k+1}, \; \G_{2k+1}',$ or
$\G_{2k-1}^{-1}\G_{2k-1}'^{-1}\G_{2k}\G_{2k-1}'\G_{2k-1},$ for $k=1, \cdots, n-1,$ $l=1, \cdots, 2k-1.$ Finally,
we also have the projective relation \begin{tiny}\begin{eqnarray} \G_{2n-1}'\G_{2n-1}\dots \G_1'\G_1 &=& e.
\label{proj}
\end{eqnarray}\end{tiny}

Now we simplify these relations.

>From (\ref{jbraids}) ($j=1$), we get $\G_1=\G_2'\G_2\G_1'\G_2^{-1}\G_2'^{-1}$; hence
$\G_2'^{-1}\G_1\G_2'=\G_2\G_1'\G_2^{-1}$. From (\ref{1}), we can rewrite this as
$\G_2'^{-1}\G_1'\G_2'=\G_2\G_1'\G_2^{-1}$.  From (\ref{j-ident}) ($j=1$), we have
\begin{equation*}
\langle \G_1',\G_2\rangle = \langle \G_1',\G_2'\rangle = e \ \ \ \mbox{and} \ \ \  \G_2'\G_1'\G_2 =
\G_1'\G_2'\G_1'^{-1}, \ \G_2\G_1'\G_2^{-1} = \G_1'\G_2\G_1'.
\end{equation*}
Thus, $\G_1'\G_2'\G_1'^{-1} = \G_1'\G_2\G_1'$. Modulo the $\G_i^2,$ this means that $\G_2=\G_2'$.

Likewise, from (\ref{ibraids}) and (\ref{i-ident}) ($i=2$), we get
$\G_3'=\G_3\G_2'\G_2\G_3\G_2^{-1}\G_2'^{-1}\G_3^{-1}$.

Now since $\G_2=\G_2'$ and $\G_2^2=e,$ we find that modulo the squares, $\G_3'=\G_3\G_3\G_3^{-1}=\G_3$.

By induction, we conclude that $\G_k'=\G_k$ modulo the squares, for $1 \leq k \leq 2n-1.$
Hence, $G / \langle \G_k^2 \rangle \cong S_{2n}.$   Hence, the kernel of the map $G / \langle \G_k^2 \rangle \rightarrow S_{2n}$ is trivial.
\end{proof}


\begin{thebibliography}{99}


\bibitem{3}
Amram, M., {\em Galois Covers of Algebraic Surfaces,} Ph.D. Dissertation, 2001.

\bibitem{pil}
Amram, M., Ciliberto, C., Miranda, R.,  Teicher, M., {\em Braid
monodromy factorization for a non-prime $K3$ surface branch
curve}, Israel Journal of Mathematics {\bf 170}, 2009, 61-94.



\bibitem{Cayley} Amram, M., Dettweiler, M., Friedman, M., Teicher, M.,
{\em Fundamental group of the complements of the Cayley's
singularities},  Beitr{\"a}ge zur Algebra und Geometrie,
Contributions to Algebra and Geometry {\bf 50}, No. 2, 2009,
469-482.



\bibitem{fr1} Amram, M., Friedman, M., Teicher, M., {\em
On the fundamental group of the complement of the branch curve of
the second Hirzebruch surface}, Topology {\bf 48}, 2009, 23-40.




\bibitem{fr2} Amram, M., Friedman, M., Teicher, M., {\em
The fundamental group of the branch curve of the complement of the
surface $\C\P^1 \times T$}, Acta Mathematica Sinica {\bf 25} (9),
2009, 1443-1458.



\bibitem{Ogata}
Amram, M.,   Ogata, S., {\em Toric varieties - degenerations and
fundamental groups}, Michigan Math. J. {\bf 54}, 2006, 587-610.



\bibitem{5}
Amram, M.,   Teicher, M.,{\em On the degeneration, regeneration
and braid monodromy of $T \times T$}, Acta Applicandae
Mathematicae {\bf 75} (1), 2003, 195-270.



\bibitem{8}
Amram, M.,  Teicher, M., Vishne, U., {\em The fundamental group of
the Galois cover of Hirzebruch surface $F_1(2,2)$ branch curve},
International J. of Algebra and Computation {\bf 17}, No. 3, 2007,
507-525.


\bibitem{TxT}
Amram, M.,  Teicher M., Vishne, U., {\em The fundamental group of
the Galois cover of the surface $T \times T$}, International J. of
Algebra and Computation {\bf 18}, no. 7, 2008, 1-24.

\bibitem{1}
Artin, E., {\em Theorie der Zoepfe}, Hamburg Abh. {\bf 4}, 1926,
47-72.

\bibitem{2}
Artin, E., {\em Theory of braids}, The Annals of  Mathematics {\bf
48}, 1947,  101-126.



\bibitem{dennis}
Auroux, D., Donaldson, S., Katzarkov, L., Yotov., M., {\em Fundamental groups of complements of plane curves and
symplectic invariants}, Topology {\bf 43}, 2004, 1285-1318.




\bibitem{9}
Chisini, O., {\em Courbes de diramation des planes multiple et
tresses algebriques}, Deuxieme Colloque de Geometrie Algebrique
tenu a Liege, CBRM, 1952, 11-27.




\bibitem{10}
Ciliberto, C., Lopez, A., Miranda, R., {\em Projective
degenerations of K3 surfaces, Gaussian maps, and Fano threefolds.}
Invent. Math. {\bf 114}, 1993, 641-667.



\bibitem{11}
Ciliberto, C., Miranda, R.,  Teicher, M., {\em Pillow degenerations of $K3$ surfaces}, Applications of Algebraic
Geometry to Coding Theory, Physics and Computations, NATO Science Series II/36, Kluwer Acad. Publish., 2001,
53-64.

\bibitem{FRT} Freitag, P., Robb, A., Teicher, M., {\em The fundamental Groups of Galois covers of Hirzebruch
surfaces,} preprint.

\bibitem{12}
Kulikov, V., {\em Degenerations of K3 surfaces and Enriques
surfaces,} Math. USSR Izvestija {\bf 11}, 1977, 957-989.


\bibitem{lidtke}
Liedtke, C., {\em Fundamental Groups of Galois Closures of Generic Projections}, Journal of AMS, 2008, 1-18.




\bibitem{Mo}
Moishezon, B., {\em Stable branch curves and braid monodromy},
Lectures Notes in Math. {\bf 862}, 1981, 107-192.



\bibitem{13}
Moishezon, B., {\em Algebraic surfaces and the arithmetic of braids, I}, Arithmetic and Geometry, papers
dedicated to I.R. Shafarevich, Birkh\"{a}user, 1983, 199-269.



\bibitem{Mo0}
Moishezon, B. {\em The arithmetic of braids and a statement of
Chisini}, Contemporary Math. {\bf 164}, 1994, 151-175.

\bibitem{MRT} Moishezon, B., Robb, A., Teicher, M., {\em On Galois covers of Hirzebruch surfaces}, Math. Ann.
{\bf 305} (1996),
 no. 3, 493-539.

\bibitem{15}
Moishezon, B.,  Teicher, M., {\em  Simply connected algebraic surfaces of positive index}, Invent. Math. {\bf
89}, 1987, 601-643.

\bibitem{GalCovs} Moishezon, B., Teicher, M., {\em Galois coverings in the theory of algebraic surfaces},
Proceedings of Symposia in Pure Mathematics  {\bf 46}, 1987,
47-65.

\bibitem{16}
Moishezon, B., Teicher, M., {\em Braid group technique in complex
geometry I, Line arrangements in $\C\P^2$}, Contemporary Math.
{\bf 78}, 1988, 425-555.


\bibitem{17}
Moishezon, B., Teicher, M., {\em Braid group technique in complex
geometry II, From arrangements of lines and conics to cuspidal
curves}, Algebraic Geometry, Lect. Notes in Math.   {\bf 1479},
1990, 131-180.


\bibitem{MoTe6}
Moishezon, B., Teicher, M., {\it Finite fundamental groups, free over $\Z/c\Z$, Galois covers of $\C \P^2$},
Math. Ann. {\bf 293}, 1992, 749-766.


\bibitem{18}
Moishezon, B., Teicher, M., {\em  Braid group technique in complex
geometry III: Projective degeneration of $V_3$,}  Contemp. Math.
{\bf 162}, 1993, 313-332.


\bibitem{19}
Moishezon, B., Teicher, M., {\em Braid group technique in complex
geometry IV: Braid monodromy of the branch curve $S_3$ of $V_3
\rightarrow \C\P^2$ and application to $\pi_1(\C^2 - S_3, \ast)$,}
Contemporary Math. {\bf 162}, 1993, 332-358.


\bibitem{MoTe10}
Moishezon, B., Teicher, M., {\it Fundamental groups of complements
of branch curves as solvable groups}, Israel Mathematics
Conference Proceedings (AMS Publications)  {\bf 9}, 1996, 329-346.

\bibitem{Robb} Robb, A., {\em On branch curves of algebraic surfaces}, Singularities and Complex Geometry
(Beijing, 1994),
193--221

\bibitem{RTV}
 Rowen, L.,  Teicher, M.,  Vishne, U., {\em Coxeter covers of the
symmetric groups}, J. Group Theory  {\bf 8}, 2005, 139-169.


\bibitem{vk}
 van Kampen, E.R., {\em On the fundamental group of an algebraic curve}, Amer. J. Math. {\bf 55}, 1933, 255-260.


\bibitem{Z1} Zariski, O., {\em On the Poincare group of rational plane curve}, Amer. J. Math., {\bf 58} 1936,
607-619.


\bibitem{zar}
Zariski, O.,  {\em On the topological discriminant group of a Riemann surface of genus p}, Amer. J. Math. 59,
1937, 335-358.


\bibitem{Z3} Zariski, O.,{\em Algebraic Surfaces}, (Ch. VIII), Second Edition, Springer, 1971.



\end{thebibliography}
\end{document}